\documentclass[twoside,11pt]{article}

\usepackage{jmlr2e}
\usepackage{amsmath}

\usepackage{caption}
\usepackage{subcaption}
\usepackage{booktabs}
\usepackage{float}

\newcommand\norm[1]{\left\lVert#1\right\rVert}

\jmlrheading{1}{2000}{1-48}{4/00}{10/00}{Merlin Mpoudeu and Bertrand Clarke}

\ShortHeadings{ VC-Dimension}{Mpoudeu and Clarke}
\firstpageno{1}

\begin{document}

\title{\bf Model Selection via the VC Dimension}

\author{\name Merlin Mpoudeu \email merlin.mpoudeu@huskers.unl.edu \\
       \addr Department of Statistics\\
       University of Nebraska-Lincoln\\
       Lincoln, NE 68503, USA
       \AND
       \name Bertrand\ Clarke \email bclarke3@unl.edu \\
       \addr Department of Statistics\\
       University of Nebraska-Lincoln\\
       Lincoln, NE 68503, USA}

\editor{Isabelle Guyon, Ingo Steinwart, Saharon Roset, John Shawe-Taylor}

\maketitle

\begin{abstract}
We derive an objective function that can be optimized to give an estimator 
of the Vapnik-Chervonenkis dimension for model selection in regression 
problems. We verify our estimator is consistent. Then, we verify it 
performs well compared to seven other model selection techniques.
We do this for a variety of types of data sets. 
\end{abstract}


\section{Complexity and Model Selection}

Model selection is often the first problem that must be addressed when analyzing data. 
In {\cal{M}}-closed problems, see \cite{Bernado2000Bayesian}, the analyst posits a list of models and assumes one of them is true. In such cases, model selection is any procedure that uses data 
to identify one of the models on the model list. There is a vast literature on model selection in this context including information based methods such as the Aikaikie Information Criterion (AIC), the Bayes information criterion (BIC), residual based methods such as Mallows $C_p$ or branch and bound, and code length methods such as the two-stage coding proposed by \cite{barron1991minimum}. We also have computational search methods such as simulated annealing and genetic algorithms. In addition, cross-validation (CV) is often used with
non-parametric methods such as recursive partitioning, 
neural networks (aka deep learning) and kernel methods.
A less well developed approach to model selection is via complexity as assessed by the Vapnik-Chervonenkis (VC) dimension, here denoted by $d_{VC}$. Its earliest usage seems
to be in \cite{Vapnik:Chervonenkis:1968}.   A translation into English was published as \cite{Vapnik:Chervonenkis:1971}.  The VC dimension was initially called an
index of a collection of sets with respect to a sample and was developed to provide sufficient conditions for the uniform convergence of empirical distributions.
It was extended to provide a sense of dimension for function spaces, particularly for functions that represented classifiers.  

After two decades of development this
formed the foundation for the field of Statistical Learning Theory.
Amongst others, Statistical Learning Theory has the key concepts of 
empirical risk minimization (ERM) and structural risk minimization (SRM). 
The idea behind ERM is to find an action e.g., an indicator function,
that minimizes the empirical risk, see \cite{Vapnik:1998}, p. 8.  
The point of SRM is to consider the solutions to the ERM problem that have the smallest 
complexity, for instance as measured by the VC dimension, see \cite{Vapnik:1998}, p. 10.
Seeking a tradeoff between error and complexity has long been a guiding principle in
statistics; see \cite{Theil:1957} who helped develop 
the `{\rm $R^2_{\sf adj}$}' often used in linear regression.

Although, the VC dimension goes back to 1968, it wasn't until \cite{Vapnik:etal:1994} that a method for estimating $d_{VC}$ was proposed in the classification context. Specifically, given a collection $\mathbb{C}$ of classifiers, \cite{Vapnik:etal:1994} tried to estimate the VC dimension of $\mathbb{C}$ by deriving an objective function based on the expected value of the maximum difference between two empirical evaluations of a single loss function, here denoted by $\Delta$. The two empirical values come from dividing a given data set into a first and second part. The objective function proposed by \cite{Vapnik:etal:1994} depends on $d_{VC}$, the sample size $n$, and several constants that had to be determined.  Using their objective function, they derived an estimator $\hat{d}_{VC}$ for $d_{VC}$ given a class $\mathbb{C}$ of classifiers. This algorithm treated possible sample sizes as design points $n_{1}, n_{2}, \cdots, n_{L}$ and requires one level of bootstrapping. Despite the remarkable contribution of \cite{Vapnik:etal:1994},
the objective function was over-complex and the algorithm did not give a tight enough bound on $\Delta$.  Later, Vapnik and his collaborators suggested a fix to tighten the bound on $\Delta$. We 
do not use this here; it is unclear if this `fix' will work in classification, let alone regression.

Choosing the design points is a nontrivial source of variability in the estimate of $d_{VC}$. So, \cite{shao2000measuring} proposed an algorithm, based on extensive simulations, to generate optimal values of $n_{1}, n_{2}, \cdots, n_{L}$, given $L$. They argued that non-uniform values of the $n_{l}$'s gave better results than the uniform $n_{l}$'s used in \cite{Vapnik:etal:1994}.

More recently, in a pioneering paper that deserves more recognition that it has received, \cite{McDonald:etal:2011} established the consistency of the \cite{Vapnik:1998} estimator $\hat{d}_{VC}$ for $d_{VC}$ in the classification context.   

The main reason the \cite{Vapnik:etal:1994} estimator for $d_{VC}$ did not become more 
widely used despite the result in \cite{McDonald:etal:2011} is, we suggest, that it was too unstable
because the objective function did not bound $\Delta$ tightly enough.  In addition, the form
of the objective function in \cite{Vapnik:etal:1994} is more complicated and less well-motivated
than our result Theorem \ref{Theo23and4}.  The reason is that the derivation
in \cite{Vapnik:etal:1994} uses conditional probabilities, one of which
goes to zero quite quickly (with $n$).  So,  it contributes negligibly to the upper bound.  Our
derivation ignores the conditioning and bounds a form of $\Delta$ that is typically larger than
that used in \cite{Vapnik:etal:1994}.  Optimizing a factor in our bound also contributes to
generating a tighter bound on the error.


Our consistency proof is a simplification of the proof of the main result \cite{McDonald:etal:2011}.  
Accordingly, we obtain a slower rate of consistency, but the probability of correct model
selection still goes to one.  As a generality, consistency results can be regarded as
special cases of probably approximately correct (PAC) bounds, 
i.e., an estimator converges to its true value in the correct probability
on the underlying measure space.  Thus, loosely speaking,
we show that $P( \| d_{VC} - \hat{d}_{VC}\| > \delta) < \epsilon$
as $n \rightarrow \infty$ For any pre-assigned $\delta, \epsilon > 0$.
More general PAC bounds replace the fixed $\delta$ and $\epsilon$ with 
$\delta = \delta_n$ and $\epsilon = \epsilon_n$.  Usually, PAC bounds are used
to express stronger results, see for instance \cite{Parrado2012} who obtained a PAC bound of the form 
$P( D(\hat{Q} \| Q) > 1/n ) < \epsilon$ where $\hat{Q}$
is an estimator of the distribution $Q$ and $D$ is the relative entropy.  
In some cases of parameter estimation with independent data $X_{1:n} = (X_1, \ldots , X_n)$ one can obtain stronger results such as 
$P_\theta (W( N(\theta, \epsilon) | X_{1:n} )^c >  e^{-\gamma_1 n}) 
< e^{- \gamma_2 n}$ for large $n$,
where $\theta$ is the true value of a parameter, $W(\cdot | X_{1:n})$
is the posterior, $N(\theta, \epsilon)$ is an $\epsilon$-neighborhood
of $\theta$, and $\gamma_1, \gamma_2 > 0$.  Such results are easier to obtain for
discrete parameters than for continuous parameters.  This may explain why \cite{McDonald:etal:2011} are able to get an exponentially small form for $\epsilon$ for the \cite{Vapnik:etal:1994} estimator in a classification context.
We have established simple consistency because this is enough for
our purposes (regression) and we were unable to extend the more sophisticated proof in \cite{McDonald:etal:2011}.

Our overall strategy is to derive an objective function for estimating $d_{VC}$ in the regression setting 
that provides, we think, a tighter bound on a modified form of $\Delta$.  To convert from classification 
to regression, we discretize the loss 
used for regression into $m$ intervals (the case $m=1$ would then apply to classification).
To get a tighter bound, we change the form of $\Delta$ from what \cite{Vapnik:etal:1994} used and 
we optimize over the leading factor in our upper bound.   To use our estimator,
we use an extra layer of bootstrapping so the quantity we empirically optimize 
represents the quantity we derive theoretically more accurately. The extra layer of bootstrapping 
stabilizes our estimator of $d_{VC}$ and appears to reduce its dependency on the $n_{l}$'s.
If the models are nested in order of increasing VC dimension, it is straightforward to choose the model 
with VC dimension closest to our estimate $\hat{d}_{VC}$.  Otherwise, we can convert a non-nested 
problem to the nested case by ordering the inclusion of the covariates using a shrinkage method 
such as the `smoothly clipped absolute deviation' (SCAD, \cite{Fan&Li}), or correlation (see \cite{fan2008sure}), 
and use our $\hat{d}_{VC}$ as before.   Even when we force a model list to be nested, our model
selection method performs well compared to a range of competitors including \cite{Vapnik:etal:1994}'s 
original method, two forms of empirical risk minimization (denoted $\widehat{ERM}_{1}$ and $\widehat{ERM}_{2}$), AIC, BIC, CV (10-fold), SCAD,
and adaptive LASSO (ALASSO, \cite{zou2006adaptive}). Our general findings indicate that in 
realistic settings, model selection via estimated VC dimension, when properly 
done, is fully competitive with existing methods and, unlike them, rarely gives abberant results.

This manuscript is structured as  follows. In Sec. \ref{chap:bounds} we present the main theory
justifying our estimator.   In Subsec. \ref{sec:Ext:Vap:Bounds} we discretize 
bounded loss functions so that upper bounds for the distinct regions of the expected 
supremal difference of empirical losses can be derived
and in Subsec. \ref{estimator} we define our estimator of the VC dimension and give an 
algorithm for how to compute it.
In Sec. \ref{ProofOfConsistency} use \cite{McDonald:etal:2011}'s consistency 
theorem to motivate our consistency theorem for $\hat{d}_{VC}$.
In Sec. \ref{chap:Numerical:Studies} we compare our method for model selection to AIC, BIC, CV, 
$\widehat{ERM}_{1}$, and $\widehat{ERM}_{2}$.    In this context, we suggest criteria to guide 
the selection of design points.   Our comparisons also include simplifying non-nested model 
lists by using correlation, SCAD, and ALASSO.
In Sec. \ref{GenreConcl} we discuss our overall findings in the context of model selection.

\section{Deriving an optimality criterion for estimating VC dimension}
\label{chap:bounds}

We are going to bound a form of $\Delta$ for use in regression.
This is the bound we will use to derive an estimator of the VC dimension.  
In Sec. \ref{sec:Ext:Vap:Bounds}, we present our alterantive version of the
\cite{Vapnik:etal:1994} bounds and in Sec. \ref{estimator}, we present an 
estimator of $d_{VC}$.  

\subsection{Extension of the \cite{Vapnik:etal:1994} bounds to regression }
\label{sec:Ext:Vap:Bounds}

Let $Z = (X,Y)$ be a random variable with outcomes $z = (x,y) $ assuming values in 
$\cal{X} \times \cal{Y}$. The first entry, $X = x$, is regarded as an  explanatory variable leading to $Y = y$.  Let $D$ be a set of $2n$ independent and identical (IID) copies of $Z$ and write $D =D^1 \cup D^2$ where $D^1 = \{Z_{1}, \ldots, Z_{n} \}$ is the first half and $D^{2} = \{Z_{n+1}, \ldots, Z_{2n}\}$ is the second half. Writing $Z_{i} = (X_{i}, Y_{i})$ for $i = 1\ldots 2n$, let 
$$
Q\left(Z_{i}, \alpha\right) = L\left(Y_{i}, f\left(X_{i}, \alpha\right)\right),
$$
for a bounded real valued loss function $L$ and $\alpha \in \Lambda$.
We assume that $\Lambda$ is a connected compact set in a finite dimensional 
real space and that the interior of $\Lambda$ is open and connected. 
Also, we assume the continuous functions $f(\cdot \mid \alpha)$ are 
parametrized by $\alpha$ continuously and one-to-one.
Thus, in our examples,
$\Lambda$ will be the parameter space for a class of regression functions $f(\cdot \mid \alpha)$.
For convenience we assume $L$, and hence $Q$, are also
continuous.

Now, for $\ell = 1, 2$ let $\alpha_\ell \in \Lambda$ and write
$$
Q_{1}\left(Z_{i}\right) = Q_{1}\left(Z_{i},\alpha_{1}\right)= L\left(Y_{i},f\left(X_{i},\alpha_{1}\right)\right)
\quad \hbox{and}\quad 
Q_{2}\left(Z_{i}\right)=Q_{2}\left(Z_{i},\alpha_{2}\right)= L\left(Y_{i},f\left(X_{i},\alpha_{2}\right)\right)
$$
where, in the first expression, it is assumed that  $Z_{i}\in D^{1}$ and in the second expression it is assumed that $Z_{i} \in D^{2}$.
Also, assume $0 \leq Q\left(\cdot, \alpha \right) \le B$ for all $\alpha \in \Lambda$ and some $B\in \mathbb{R}$.

Denote the discretization of $Q_{\ell}$ using $m$ disjoint intervals (with union $[0, B)$) given by
 \begin{equation}
 \label{E11}
 Q_{\ell}(Z) = Q^{m}_{\ell j}\left(z_{i},\alpha_{\ell}\right)=\sum_{j = 0}^{m-1}\frac{(2j+1)B}{2m}\mathbb{I}\left[Q_{\ell}(Z)\in I_{j}^{m}\right],
\end{equation}
where $\ell = 1, 2$, $z_{i} \in D^{\ell}$ and $j = 0, 1,\cdots, m-1$. 
Note that the numbers $\frac{(2j+1)B}{2m}$ are the midpoints of the uniform left-closed, right-open
partition of $[0,B)$ into $m$ sub-intervals, here denoted $I^m_j$. In \eqref{E11},
$\mathbb{I}[\cdot]$ is an indicator function taking value 1 when its argument is true and taking value 0 when its argument is false. 
Now, let
\begin{equation}
    \label{nu1and2}
    \nu_{1}\left(D^{1}\right)= \frac{1}{n}\sum_{i=1}^{n}Q_1\left(Z_{i}\right)
\quad
\hbox{and}
\quad
\nu_{2}\left(D^{2}\right)= \frac{1}{n}\sum_{i=1}^{n}Q_2\left(Z_{n+i}\right)
\end{equation}
be the empirical risk using the first and second half of the sample,
respectively. 
Observe that the empirical counts of the data points whose losses land in $I^m_j$ are 
$$
N^m_{1j} = \sum_{i=1}^{n}\mathbb{I}\left[Q_{1}\left(Z_{n+i}\right)\in I_j^m\right]
\quad \hbox{and} \quad
N_{2j}^{m} = \sum_{i=1}^{n}\mathbb{I}\left[Q_{2}\left(Z_{i}\right)\in I_j^m\right].
$$
This means we are evaluating the $\alpha_{1}$ model on the second half of the data and the $\alpha_{2}$ model on the first half of the data. 
So, we have expressions for the empirical losses of the discretized loss function:
\begin{equation}
\label{nu1and2discretized}
\nu_{1}^{m}\left(D^2\right) = \frac{1}{n}\sum_{j = 0}^{m - 1}N_{2,j}^{m}\frac{(2j+1)B}{2m} \quad \hbox{and} \quad \nu_{2}^{m}\left(D^1\right) = \frac{1}{n}\sum_{j = 0}^{m - 1}N_{1,j}^{m}\frac{(2j+1)B}{2m}.
\end{equation}

We begin by controlling $\Delta$, the expected supremal difference between two evaluations of a bounded loss function, to be formally defined in \eqref{delta}.
Let $\epsilon > 0$ and let 
\begin{equation}
    \label{Intervalnu1and2}
    \nu_{1,j}^m\left(Z_{n+1:2n}\right) = \frac{1}{n}N_{2,j}^{m}\frac{(2j+1)B}{2m} \quad
    \hbox{and} \quad
    \nu_{2,j}^m\left(Z_{1:n}\right) = \frac{1}{n}N_{1,j}^{m}\frac{(2j+1)B}{2m}.
\end{equation}
We define the events
\begin{equation}
    \label{Aepsilon}
    A_{\epsilon} = \bigcup_{m = 0}^{m-1}A_{\epsilon, m},
\end{equation}
where
\begin{equation}
    \label{NewAepsilon}
    A_{\epsilon,m} = \left\{Z_{1:2n} | \sup_{\alpha_1, \alpha_{2} \in \Lambda} \left[\nu^m_1\left(Z_{n+1:2n}\right) - \nu^m_2\left(Z_{1:n}\right)\right]\geq \epsilon\right\}.
\end{equation}
Since $A_{\epsilon}$ is defined on the entire range of our loss function, and we want to partition the range into $m$ disjoint intervals, write
\begin{eqnarray*}
   A_{\epsilon,m}\subseteq \left\{Z_{1:2n}~|~ \exists~ j : \sup_{\alpha_{1},~\alpha_{2}\in ~ \Lambda}\left(\nu_{1j}^{m}\left(Z_{n+1:2n}\right)-\nu_{2j}^{m}\left(Z_{1:n},\alpha_{2},m\right)\right)\ge\frac{\epsilon}{m}\right\}
   \subseteq \bigcup_{j=0}^{m-1} A_{\epsilon,m,j},
\end{eqnarray*}
where $ A_{\epsilon,m,j} = \left\{Z_{1:2n}~|~ \sup_{\alpha_{1},~\alpha_{2}~\in\Lambda}\left(\nu_{1j}^{m}\left(Z_{n+1:2n}\right)-
\nu_{2j}^{m}\left(Z_{1:n}\right)\right)\ge\frac{\epsilon}{m}\right\}$.
Since $\Lambda$ is compact, $\Lambda \times \Lambda$ is compact, and the supremum over $\Lambda \times \Lambda$ in $A_{\epsilon,m,j}$ will be achieved at some 
$$
\left(\alpha_{1j}^{m}, \alpha_{2j}^{m}\right)  = \arg\sup_{\alpha_{1},~\alpha_{2}\in \Lambda}\left(\nu_{1j}^{m}\left(z_{n+1:2n}\right)-\nu_{2j}^{m}\left(z_{1:n}\right)\right)\in \Lambda \times \Lambda
$$
where $z_{n+1:2n}$ and $z_{1:n}$ are outcomes of $Z_{n+1:2n}$ and $Z_{1:n}$ respectively.

Next, fix any value $j \in \{0, 1, \ldots, m-1\}$. 
For any fixed $z_{2n}$, and any given $\alpha_{1}$ and $\alpha_{2}$, form the vector of length $2n$ of the form
\begin{eqnarray}
&& \left(Q_{1j}^m\left(Z_{n+1:2n}, \alpha_{1}\right), Q_{2j}^m\left(Z_{1:n}, \alpha_{2}\right)\right) \nonumber \\  \nonumber &=&  \left(Q^m_{1j}\left(Z_{n+1}, \alpha_{1}\right), \ldots,
Q^m_{1j}\left(Z_{n+1}, \alpha_{1}\right), Q^m_{2j}\left(Z_{1}, \alpha_{2}\right), \ldots, Q^m_{2j}\left(Z_{n}, \alpha_{2}\right)\right).
\end{eqnarray}
For any $\alpha_{1}$, $\alpha_{1}^\prime$, $\alpha_{2}$, and $\alpha_2^\prime\in \Lambda$, define
\begin{eqnarray}
&& 
\left(\alpha_{1}, \alpha_2\right) \sim \left(\alpha_1^\prime, \alpha^\prime_{2}\right) \Longleftrightarrow
\nonumber \\
&& \left(Q_{1j}^m\left(Z_{n+1:2n}, \alpha_{1}\right), Q_{2j}^m\left(Z_{1:n}, \alpha_{2}\right)\right) = \left(Q_{1j}^m\left(Z_{n+1:2n}, \alpha^\prime_{1}\right), Q_{2j}^m\left(Z_{1:n}, \alpha^\prime_{2}\right)\right).
\nonumber
\end{eqnarray}
So, for any fixed $Z_{1:2n} = z_{1:2n}$ it is seen that $\sim$ is an equivalence relation on $\Lambda \times \Lambda$
and therefore partitions $\Lambda \times \Lambda$ into disjoint equivalence classes.
Denote the number of these classes by $N^{\Lambda}_{j} = N^{\Lambda}_{j}\left(z_{1:2n}\right) = N^{\Lambda}_{j}(z_{1},z_{2},\dots,z_{2n})$.

For fixed $j$, let $K_{j}\left(z_{1:2n}\right)$ be the number of equivalence classes in $\Lambda \times \Lambda$ and for a given $\alpha^* = (\alpha_{1}, \alpha_{2})\in \Lambda \times \Lambda $
write $[\alpha^*]$ to be the equivalence class that contains it. Now, for $k =1 \ldots,K_j$ let
$$
\alpha_{jk}^{*} = \left(\alpha^{*}_{1jk}, \alpha^{*}_{2jk}\right) = \arg\sup_{\alpha_{1j},\alpha_{2j}\in\left(\Lambda \times\Lambda\right)_{k}}\left(\nu_{1j}^{m}\left(Z_{n+1:2n}\right)-\nu_{2j}^{m}\left(Z_{1:n}\right)\right) ,
$$
where $(\Lambda \times \Lambda )_{k}$ is the $k$-th equivalence class.
Now, $\bigcup_{k=1}^{K_{j}}\left[\alpha_{jk}^{*}\right] = \Lambda \times \Lambda$ and  $\left[\alpha^*_{jk}\right]\cap\left[\alpha^*_{jk^\prime}\right] = \phi$
unless $k = k^\prime$. 
Clearly, $K_{j}= N^{\Lambda}_{j}\left(z_{1:2n}\right)$.

To make use of the above partitioning of $A_{\epsilon}$,
consider mapping the underlying measure space space ${\cal{Z}}_{1:2n}$ onto itself using $(2n)!$ distinct permutations
$T = T_{i}$.  Then, if $f$ is integrable with respect to the distribution of $Z_i$, its Riemann-Stieltjes
integral satisfies
$$
\int_{{\cal{Z}}^{2n}}^{}f\left(Z_{1:2n}\right)\emph{dF}\left(Z_{1:2n}\right) = \int_{{\cal{Z}}^{2n}}^{}f\left(T_{i}Z_{1:2n}\right)\emph{dF}\left(Z_{1:2n}\right),
$$
and this gives
\begin{equation}
\label{permid}
  \int_{{\cal{Z}}^{2n}}^{}f\left(Z_{1:2n}\right)\emph{dF}\left(Z_{1:2n}\right) = \int_{{\cal{Z}}^{2n}}^{}\frac{\sum_{i=1}^{(2n)!}f\left(T_{i}Z_{1:2n}\right)}{(2n)!}\emph{dF}\left(Z_{1:2n}\right).
\end{equation}
Our first main result is structured the same way as in \cite{Vapnik:etal:1994}, however, 
the differences are in the details. For instance, our equivalence class is defined on 
$\Lambda \times \Lambda$, we use a cross-validation form of the error, 
we discretized the loss function, and our result leads to Theorem \ref{Theo23and4} that only has one term, whereas the corresponding result in \cite{Vapnik:etal:1994} has three terms.
 \begin{theorem}
   \label{NewProp2}
  Let $\epsilon \ge 0,\quad m \in  \mathcal{N}, \hbox{and},\quad d_{VC} = VC\left( \left\{Q\left(\cdot,\alpha\right): \alpha \in  \Lambda\right\}\right)$. If $d_{VC}$ is finite, then
  \begin{equation}
  \label{NewUPBO}
  P\left(A_{\epsilon,m}\right) \leq 2m\left(\frac{2ne}{d_{VC}}\right)^{d_{VC}}\exp\left\{-\frac{n\epsilon^2}{m^{2}}\right\}.
  \end{equation}
 \end{theorem}

\noindent
{\bf Remark:}
 The technique used to prove \eqref{NewUPBO} is similar to the proof of Theorem 4.1 in \cite{Vapnik:1998}
 giving bounds for the uniform convergence of the empirical risk. The hypotheses of Theorem 4.1 in 
\cite{Vapnik:1998} require only the existence of the key quantities e.g the annealed entropy, and the 
growth function. Our only extra condition is that $d_{VC}$ is finite.

\begin{proof}
Let $\Delta^{m}_{j}\left(T_{i}Z_{1:2n},\alpha_{j}^{*}\right)=\nu_{1j}^{m}\left(T_{i}Z_{n+1:2n},\alpha_{1j}^{*}\right) - \nu_{2j}^{m}\left(T_{i}Z_{1:n},\alpha_{2j}^{*}\right)$, 
where

\begin{eqnarray}
\alpha_{j}^{*} = \left(\alpha_{1j}^*, \alpha_{2j}^*\right)=  \arg\underset{\left(\alpha_{1j}, \alpha_{2j}\right) \in \Lambda \times \Lambda} \max \Delta_{j}^{m}\left(T_{i}Z_{1:2n},\alpha_{j}\right) \in \Lambda \times \Lambda.
\nonumber
\end{eqnarray}
Using some manipulations, we have
  \begin{eqnarray}
  \label{UpAep}
  P\left(A_{\epsilon,m}\right) &\leq& P\left(\bigcup_{j=0}^{m-1}A_{\epsilon,m,j}\right)\leq \sum_{j=0}^{m-1}P\left(A_{\epsilon,m,j}\right)\nonumber\\\nonumber\\
   &=& \sum_{j=0}^{m-1}P\left(\left\{Z_{1:2n}:\sup_{\alpha_{1},~\alpha_{2}\in\Lambda}\left(\nu_{1j}^{m}(Z_{n+1:2n},\alpha_{1})-\nu_{2j}^{m}(Z_{1n},\alpha_{2})\right)\geq\frac{\epsilon}{m}\right\}\right).
\nonumber
\end{eqnarray}
Continuing the equality gives that the RHS equals
\begin{eqnarray}
 \sum_{j=0}^{m-1}P\left(\left\{Z_{1:2n}:\left(\nu_{1j}^{m}(T_{i}Z_{n+1:2n},\alpha_{1j}^{*})-\nu_{2j}^{m}(T_{i}Z_{1:n},\alpha_{2j}^{*})\right)\geq\frac{\epsilon}{m}\right\}\right)
\nonumber
\end{eqnarray}
\begin{eqnarray}
   &=& \sum_{j=0}^{m-1}P\left(\left\{Z_{1:2n}: \Delta_{j}^{m}\left(T_{i}Z_{1:2n},\alpha_{j}^{*}\right)\geq\frac{\epsilon}{m}\right\}\right)\nonumber\\
   &=& \frac{1}{(2n)!}\sum_{j=0}^{m-1}\sum_{i=1}^{(2n)!}P\left(\left\{Z_{1:2n}: \Delta_{j}^{m}\left(T_{i}Z_{1:2n},\alpha_{j}^{*}\right)\geq\frac{\epsilon}{m}\right\}\right)
\nonumber \\
   &=& \frac{1}{(2n)!}\sum_{j=0}^{m-1}\sum_{i=1}^{(2n)!}\int_{}^{}\emph{I}_{\left\{Z_{1:2n}: \Delta_{j}^{m}\left(T_{i}Z_{1:2n},\alpha_{j}^{*}\right)\geq\frac{\epsilon}{m}\right\}}\left(z^{2n}\right)\emph{dP}(Z_{1:2n}).
\label{permprobbdd}
\end{eqnarray}
Using the properties of the equivalence relation $\sim$, for each fixed $j$ and $Z_{1:2n}$ we have the inequality 
\begin{eqnarray}
\label{Indfnbdd}
  \emph{I}_{\left\{Z_{1:2n}:\Delta^m_j\left(T_{i}Z_{1:2n},\alpha^{*}_{j}\right)\geq\epsilon\right\}}(Z_{1:2n})
&\leq&
\emph{I}_{\left\{Z_{1:2n}: \Delta_{j}^{m}\left(T_{i}Z_{1:2n},\alpha^{*}_{j1},m\right)\ge \frac{\epsilon}{m}\right\}}\left(z_{1:2n}\right) \nonumber \\
 &+&
\dots + \emph{I}_{\left\{Z_{2n}:\Delta_{j}^{m}\left(T_{i}Z_{1:2n},\alpha^{*}_{jN^{\Lambda}_{j}(z_{1:2n})}\right)\ge\frac{\epsilon}{m}\right\}}\left(z_{1:2n}\right) \nonumber\\
   &=& \sum_{k=1}^{N^{\Lambda}_{j}(Z_{1:2n})}\emph{I}_{\left\{Z_{1:2n}: \Delta_{j}^{m}\left(T_{i}Z_{1:2n},\alpha^{*}_{jk}\right)\ge\frac{\epsilon}{m}\right\}}\left(Z_{1:2n}\right)
\end{eqnarray}
where
\begin{eqnarray*}
  A_{\epsilon,m,j,k} &=& \left\{Z_{1:2n}: \Delta_{j}^{m}\left(T_{i}Z_{1:2n},\alpha^{*}_{kj}\right)\ge\frac{\epsilon}{m}\right\} \\
   &=& \left\{Z_{1:2n}: \sup_{\left(\alpha_{1j}, \alpha_{2j}\right)\in\left(\Lambda \times \Lambda\right)_{k}}\left(\nu_{1j}^{m}(T_{i}Z_{n+1:2n},\alpha_{1j})-\nu_{2j}^{m}(T_{i}Z_{1:n},\alpha_{2j})\right)\ge\frac{\epsilon}{m}\right\}
\end{eqnarray*}
and for fixed $j$ and each $k$, $\left[\alpha_{jk}^*\right] = \left(\Lambda \times \Lambda\right)_{k}$.
Now, using \eqref{Indfnbdd}, \eqref{permprobbdd} is bounded by
\begin{eqnarray}
  P\left(A_{\epsilon,m}\right) &\leq& \frac{1}{(2n)!}\sum_{j=0}^{m-1}\sum_{i=1}^{(2n)!}\int_{}^{}\sum_{k=1}^{N_{j}^{\Lambda}(Z_{1:2n})}\emph{I}_{\left\{Z_{1:2n}:\Delta_{j}^{m}(T_{i}Z_{1:2n},\alpha^{*}_{kj})\ge\frac{\epsilon}{m}\right\}}\left(Z_{1:2n}\right) \emph{dP}\left(Z_{1:2n}\right)
\nonumber\\
   &=& \int_{}^{}\sum_{j=0}^{m-1}\sum_{k=1}^{N_{j}^{\Lambda}(Z_{1:2n})}\left[\frac{1}{(2n)!}\sum_{i=1}^{(2n)!}\emph{I}_{\left\{Z_{1:2n}:\Delta_{j}^{*}(T_{i}Z_{1:2n},\alpha_{kj}^{*})\geq \frac{\epsilon}{m}\right\}}\left(Z_{1:2n}\right)\right]\emph{dP}\left(Z_{1:2n}\right) \nonumber .
\label{overallprobbdd}
\end{eqnarray}
The expression in square brackets is the fraction of the number of the $(2n)!$ permutations $T_{i}$ of $Z_{1:2n}$ for which $A_{\epsilon,m,j,k}$ is closed under $T_{i}$ for any fixed equivalence class $\left(\Lambda \times \Lambda\right)_k$. Following \cite{Vapnik:1998} Sec. 4.13, it equals
\begin{equation}
\label{E77}
  \Gamma_{j} = \sum_{k}^{}\frac{\binom{N^{m}_{j}}{k} \binom{2n-N^{m}_{j}}{N^{m}_{j}-k}}{\binom{2n}{n}}
\nonumber
\end{equation}
where the summation is over $k$'s in the set
$$
\left\{k: \left|\frac{k}{n}-\frac{N^m_j -k}{n}\right|\ge\frac{\epsilon}{m}\right\} \cap
   \{ k :  \max(0, N_j^m -n) \leq k \leq \min(N_j^m, n) \}
$$
where 
$
N^{m}_{j} = N_{1j}^{m} + N_{2j}^{m}.
$
Here, $\Gamma_{j}$ is the probability of choosing exactly $k$ sample data points whose losses fall in interval $I_{j}$.
From Sec. 4.13 in \cite{Vapnik:1998}, we have $\Gamma_{j} \le 2\exp\left(-\frac{n\epsilon^{2}}{m^2}\right)$. So, using this in \eqref{overallprobbdd} gives that $P\left(A_{\epsilon, m}\right)$ is upper bounded by
\begin{eqnarray}
   && \int_{}^{}\sum_{j=0}^{m-1}\sum_{k=1}^{N_{j}^{\Lambda}(Z_{1:2n})}2\exp\left(-\frac{n\epsilon^{2}}{m^2}\right)\emph{dP}(Z_{1:2n})
   = 2\exp\left(-\frac{n\epsilon^{2}}{m^2}\right)\int_{}^{}\sum_{j=0}^{m-1}\sum_{k=1}^{N_{j}^{\Lambda}(z^{2n})}\emph{dP}(Z_{1:2n})
\nonumber \\
   &=& 2\exp\left(-\frac{n\epsilon^{2}}{m^2}\right)\sum_{j=0}^{m-1}\int_{}^{}\sum_{k=1}^{N_{j}^{\Lambda}(Z_{1:2n})}\emph{dP}(Z_{1:2n})
   = 2\exp\left(-\frac{n\epsilon^{2}}{m^2}\right)\sum_{j=0}^{m-1}\int_{}^{}N_{j}^{\Lambda}(Z_{1:2n})\emph{dP}(Z_{1:2n})
\nonumber \\
   &=& 2\exp\left(-\frac{n\epsilon^{2}}{m^2}\right)\sum_{j=0}^{m-1}E\left(N_{j}^{\Lambda}(Z_{1:2n})\right).
\label{probbddd}
\end{eqnarray}
Recall from \cite{Vapnik:1998} that the annealed entropy is defined as
$$
H_{ann}^\Lambda(2n) = \ln{EN^\Lambda\left(z_{1},\ldots, z_{2n}\right)},
$$
and the growth function is defined as
$$
G^\Lambda(2n) =\ln{\sup_{z_{1},z_{2}, \ldots, z_{2n}}N^\Lambda\left(z_{1}, \ldots, z_{2n}\right)}.
$$
These two quantities satisfy
$$
H^\Lambda_{ann}(2n) \leq G^\Lambda(2n).
$$
Now, Theorem 4.3 from \cite{Vapnik:1998} p.145 gives
$$
H_{ann}(2n) = \ln\left(E\left(N_{j}^{\Lambda}(z_{1:2n})\right)\right)\leq G(2n)\leq d_{VC}\ln\left(\frac{2ne}{d_{VC}}\right) \Rightarrow
 E\left(N_{j}^{\Lambda}(z^{2n})\right)\leq \left(\frac{2ne}{d_{VC}}\right)^{d_{VC}}.
$$
Using this $m$ times in \eqref{probbddd} gives the  Theorem.
\end{proof}

We can use Theorem \ref{NewProp2} to give an upper bound on the unknown true risk via the following
propositions.
 Let $R\left(\alpha_{k}\right)$ be the true unknown risk at $\alpha_{k}$ and $R_{emp}\left(\alpha_{k}\right)$ be the empirical risk at $\alpha_{k}$.
\begin{proposition}
\label{subProp1}
   For any $\eta \in (0,1)$, with probability at least $1-\eta$, the inequality
  \begin{eqnarray}
  \label{EquivProp1}
  R\left(\alpha_{k}\right) \leq R_{emp}\left(\alpha_{k}\right) + m\sqrt{\frac{1}{n}\log\left(\left(\frac{2m}{\eta}\right)\left(\frac{2ne}{d_{VC}}\right)^{d_{VC}}\right)}
  \end{eqnarray}
  holds simultaneously for all functions $Q\left(z,\alpha_{k}\right)$, $k = 1,2,\cdots, K$.
\end{proposition}
This inequality follows from the additive Chernoff bound and suggests that the best model will be the one that minimizes the RHS of inequality \eqref{EquivProp1}.  The use of  \eqref{EquivProp1}
in model selection as a form of risk minimization because as 
$d_{VC}$ increases the second term on the right increases.  This limits the size of
$d_{VC}$; we denoted this technique by $ERM_{1}$ since a penalized empirical risk is being
minimized.  (It would also be accurate to call this structural risk minimization but it
is more standard amongst people who work in model selection in both computer science and
statistics to regard it as a form of ERM.)

\begin{proof}
  To obtain inequality \eqref{EquivProp1}, we equate the RHS of  Proposition \ref{NewProp2} to a positive number $0\le \eta\leq 1$. Thus:
\begin{equation*}
  \eta = 2m\left(\frac{2ne}{d_{VC}}\right)^{d_{VC}}\exp\left(-\frac{n\epsilon^{2}}{m^{2}}\right).
\end{equation*}
Solving for $\epsilon$ gives
\begin{eqnarray}
\label{solution}
  \epsilon &=& m\sqrt{\frac{1}{n}\log\left(\left(\frac{2m}{\eta}\right)\left(\frac{2ne}{d_{VC}}\right)^{d_{VC}}\right)}.
\end{eqnarray}
 Proposition \ref{subProp1} can be obtained from the additive Chernoff bounds, expression $4.4$ in \cite{Vapnik:1998} as follows
\begin{eqnarray}
\label{proofsubProp1}
R\left(\alpha_{k}\right) &\leq& R_{emp}\left(\alpha_{k}\right) + \epsilon.
\end{eqnarray}
Using \eqref{solution} in inequality \eqref{proofsubProp1}, completes the proof.
\end{proof}

Parallel to Prop. \ref{subProp1}, we have the following for the multiplicative case.

\begin{proposition}
\label{subProp2}
  For any $\eta \in (0,1)$, with probability $1-\eta$, the inequality
\begin{equation}\label{EquivalentProp2}
  R\left(\alpha_{k}\right)\le R_{emp}\left(\alpha_{k}\right) + \frac{m^{2}}{2n}\log\left(\frac{2m}{\eta}\left(\frac{2ne}{d_{VC}}\right)^{d_{VC}}\right)\left(1+\sqrt{1+\frac{4nR_{emp}\left(\alpha_{k}\right)}{m^{2}\log\left(\frac{2m}{\eta}\left(\frac{2ne}{d_{VC}}\right)^{d_{VC}}\right)}}\right)
\end{equation}
holds simultaneously for all $K$ functions in the set $Q\left(z, \alpha_{k}\right)$, $k = 1,2,\dots, K$.
\end{proposition}
This follows from the multiplicative Chernoff bound and suggests that the best model will be the one that minimizes the right hand side (RHS) of \eqref{EquivalentProp2}. Analogous to \eqref{equivProp1} 
we refer to the use of \eqref{EquivalentProp2} in model selection as $ERM_{2}$.

\begin{proof} Let $\epsilon, \eta >0$.  Then, inequality (4.18) in \cite{Vapnik:1998} gives, with probability at least $1-\eta$, that
\begin{eqnarray*}
  \frac{R\left(\alpha_{k}\right)-R_{emp}\left(\alpha_{k}\right)}{\sqrt{R\left(\alpha_{k}\right)}} \leq \epsilon.
  \end{eqnarray*}
  Routine algebraic manipulations and completing the square give
  \begin{eqnarray*}
    \left(R\left(\alpha_{k}\right)-0.5\left(\epsilon^{2}+2R_{emp}\left(\alpha_{k}\right)\right)\right)^{2}-0.25\left(\epsilon^{2}+2R_{emp}\left(\alpha_{k}\right)\right)^{2} \leq -R^{2}_{emp}\left(\alpha_{k}\right).
    \end{eqnarray*}
    Taking the square root on both sides and re-arranging gives
  \begin{eqnarray*}
  R\left(\alpha_{k}\right) \leq
R_{emp}\left(\alpha_{k}\right) + 0.5\epsilon^{2}\left(1 + \sqrt{1 + \frac{4R_{emp}\left(\alpha_{k}\right)}{\epsilon^{2}}}\right)
\end{eqnarray*}
Using \eqref{solution} in the last inequality completes the proof of the Proposition.
\end{proof}
The proof of Propositions \ref{subProp1} and \ref{subProp2} are easy and can be found in \cite{Merlin:etal:2017}.
\noindent

Formally, let 
\begin{equation}
\label{deltam}
    \Delta_{m} =  E\left(\sup_{\alpha_{1}, \alpha_{2}\in\Lambda}\left|\nu_{1}^{m}(Z_{n+1:2n},\alpha_{1})-\nu_{2}^{m}(Z_{1:n},\alpha_{2})\right|\right)
\end{equation}
and 
\begin{equation}
    \label{delta}
    \Delta =  E\left(\sup_{\alpha_{1}, \alpha_{2}\in\Lambda}\left|\nu_{1}\left(Z_{n+1:2n},\alpha_{1}\right)-\nu_{2}\left(Z_{1:n},\alpha_{2}\right)\right|\right).
\end{equation}
 Obviously, $\Delta_{m} \approx \Delta$ provided that $m, n, \hbox{and}~  d_{VC} \rightarrow \infty$ at appropriate rates and the argument of $\Delta_{m}$ satisfies appropriate uniform integrability conditions. In fact, we do not use $\Delta_{m} \rightarrow \Delta$. For our purpose, the following bounds are sufficient. They are important to our
methodology because they bound the expected maximum difference between two values of the empirical losses by an expression that can be used to estimate the VC dimension.
\begin{theorem}
\label{Theo23and4}
  \begin{enumerate}
    \item If $d_{VC}< \infty$, we have
    \begin{equation}
     \Delta_{m} \leq m\sqrt{\frac{1}{n}\ln\left(2m^{3}\left(\frac{2ne}{d_{VC}}\right)^{d_{VC}}\right)}\\ + \frac{1}{m\sqrt{n\ln\left(2m^3\left(\frac{2ne}{d_{VC}}\right)^{d_{VC}}\right)}}
    \end{equation}
    \item If $d_{VC} \le \infty$, and
  $$
  D_{p}\left(\alpha\right) = \int_{0}^{\infty}\sqrt[p]{P\left\{Q\left(z,\alpha\right)\geq c\right\}}dc \leq \infty
  $$ where $1< p\leq 2$ is some fixed parameter, we have
    \begin{equation}
      \Delta \leq \frac{D_{p}(\alpha^{*})2^{2.5+\frac{1}{p}}\sqrt{d_{VC}\ln\left(\frac{ne}{d_{VC}}\right)}}{n^{1-\frac{1}{p}}} + \frac{16D_{p}(\alpha^{*})2^{2.5+\frac{1}{p}}}{n^{1-\frac{1}{p}}\sqrt{d_{VC}\ln\left(\frac{ne}{d_{VC}}\right)}}.
    \label{Hannbd}
    \end{equation}
    \item Assume that $d_{VC} \rightarrow \infty$, $\frac{n}{d_{VC}} \rightarrow \infty$, $m \rightarrow \infty$, $\ln\left(m\right)=o(n)$, \hbox{and}
$$
D_{p}\left(\alpha\right) = \int_{0}^{\infty}\sqrt[p]{P\left\{Q(z,\alpha)\ge c\right\}} dc\leq \infty
$$
where $p=2$. Then we have that
\begin{equation}
\label{objfn}
  \Delta \leq \min\left(1,8D_{p}(\alpha^{*})\right)\sqrt{\frac{d_{VC}}{n}\ln\left(\frac{2ne}{d_{VC}}\right)} .
\end{equation}

  \end{enumerate}
\end{theorem}

\begin{proof}
The proof of Theorem \ref{Theo23and4} can be found in \cite{Merlin:etal:2017} in Appendices A1--A3. It rests on using the integral of probabilities identity and the bound in Theorem \ref{NewProp2}.
\end{proof}

\subsection{An Estimator of the VC Dimension}
\label{estimator}

The upper bound from Theorem \ref{Theo23and4} can be written as
\begin{equation}
\label{UPPTHEO6}
  \Phi_{d_{VC}}(n)=\min\left(1,8D_{p}(\alpha^{*})\right)\sqrt{\frac{d_{VC}}{n}\log\left(\frac{2ne}{d_{VC}}\right)}.
\end{equation}
This expression is meaningfully different from the form derived in \cite{Vapnik:etal:1994} and studied in \cite{McDonald:etal:2011}. Moreover, although  $\min\left(1,8D_{p}(\alpha^{*})\right)$ does not affect the optimization, it might not be the best constant for the inequality in \eqref{objfn}. So, we replace it with an arbitrary constant $c$ over which we optimize to make our upper bound as tight as possible.
In our computations,
we let $c$ vary from $ 0.01 $ to $100 $ in steps of size $0.01$. However, we have observed in practice
that the best value of $\hat{c}$ is usually between $1$ and $8$. The technique that we use to 
estimate $\hat{d}_{VC}$ is also different from that in \cite{Vapnik:etal:1994}. Indeed, 
our Algorithm \ref{Algo1} below accurately encapsulates the way the LHS of  \eqref{objfn} is 
formed unlike the algorithm in \cite{Vapnik:etal:1994}.

In particular, we use two bootstrapping procedures, one as a proxy for calculating expectations and  the second as a proxy for calculating a maximum.
Moreover, we split the dataset into two subsets.  Using the first dataset, we fit model I and using
the second we fit model II.
To explain how we find our estimate of the RHS of \eqref{objfn} from Theorem \ref{Theo23and4}, we start 
by replacing the sample size $n$ in \eqref{UPPTHEO6} with a specified value of 
design point, so that the only unknown is $d_{VC}$. Thus, formally, we replace \eqref{UPPTHEO6} by
$$
\Phi_{d_{VC}}^{*}(n_{l}) = \widehat{c}\sqrt{\frac{d_{VC}}{n_{l}}\log\left(\frac{2n_{l}e}{d_{VC}}\right)},
$$
where $\hat{c}$ is the optimal data driven constant.
If we knew the left hand side (LHS) of \eqref{objfn}, even computationally, we could use it to estimate $d_{VC}$. However, in general we don't know the LHS of \eqref{objfn}.
Instead, we generate one observation of the form
\begin{equation}
    \xi\left(n_{l}\right)= \Phi_{d_{VC}}^{*}(n_{l}) + \epsilon(n_{l})
\label{regmodel}
\end{equation}
for each design point $n_{l}$ by bootstrapping and denoted the realized values by $\hat{\xi}\left(n_{l}\right)$. In \eqref{regmodel}, we assume $\epsilon(n_{l})$ has a mean zero,
but an otherwise unknown, distribution.  We can therefore obtain a list of values of $\widehat{\xi}(n_{l})$
for the elements of $N_{L}$. In effect, we are assuming that $\Phi_{d_{VC}}^{*}(n_{l})$ provides a tight bound on $\Delta$, and hence $\Delta_{m}$ as suggested by Theorem \ref{Theo23and4}. Our algorithm is as follows.
\begin{description}
\label{Algo1}
  \item[Algorithm $\#1$:] {\bf Inputs:} A collection of regression models $\mathcal{G}=\left\{g_{\beta}:\beta\in\beta\right\}$,
  a data set, two integers $b_{1}$ and $b_{2}$ for the number of bootstrap samples, Integer $m$ for the number of disjoint intervals use to discretize the losses,
  A set of  design points $N_{L} = \left\{n_{1}, n_{2}, \cdots, n_{l}\right\}$.
  \item[1] For each $l = 1, 2,\cdots, L$ do;
  \item[2] Take a bootstrap sample of size $2n_{l}$ (with replacement) from our dataset;
  \item[3] Randomly divide the bootstrap data into two groups $~G_{1}~ \hbox{and}~ G_{2}$ of size $n_{l}$ each;
  \item[4] Fit two models one for $G_{1}$ and one for $G_{2}$;
  \item[5] Compute the square error of each model using the covariate and the  response from the other model, thus: For instance $$ SE_{1} = \left(predict(Model_{1}, x_{2})-y_{2}\right)^2,~ \hbox{and}~  SE_{2} = \left(predict(Model_{2}, x_{1})-y_{1}\right)^2$$
  where $(x_1,y_1)$ ranges over $G_1$ and $(x_2, y_2)$ ranges over $G_2$. So, there are $n_l$ values of $SE_1$ and $n_l$ values of $SE_2$.
  \item[6]  Discretize the loss function i.e., put each $SE_{1} ~ \hbox{and} ~ SE_{2}$ in one of the $m$ disjoint intervals;
  \item[7] Estimate $\nu_{1j}^{m}(Z_{n_{l}+1:2n_{l}})$ and $\nu_{2j}^{m}(Z_{1:n_{l}})$ using the $SE_{1}$'s and $SE_{2}$'s respectively in the intervals $j = 0,1, \ldots m-1$;
  \item[8] Compute the differences $\left|\nu_{1j}^{m}(z_{n_{l}+1:2n_{l}}) - \nu_{2j}^{m}(z_{1:n_{l}})\right|$ for $j = 0,1,\ldots, m-1$;
  \item[9] Repeat Steps $1-8$ $b_{1}$ times, take the mean interval-wise and sum it across all intervals so we have:
  $$
  r_{b_1}\left(n_{l}\right) = \sum_{j=0}^{m-1}mean\left|\nu_{1j}^{m}(z_{n_{1}+1:2n_{l}}) - \nu_{2j}^{m}(z_{1:n_{l}})\right|;
  $$
  \item[10] Repeat Steps $1-9$ $b_{2}$ times to $r_{b_{1,i}}(n_l)$  for $i =1,\ldots, b_{2}$ and form
   $$
  \hat{\xi}(n_{l})= \frac{1}{b_{2}}\sum_{i=1}^{b_{2}}r_{b_1,i}(n_{l}).
  $$
\end{description}

It is seen that Step 9 uses a mean even though the definition of $\Delta_m$ and $\Delta$
(see \eqref{deltam} and \eqref{delta}) has a supremum inside the expectation.  This is intentional because
using a supremum within each interval gave a worse estimator.  We suggest that summing the mean over
the intervals performs well because it is not too far from the supermum and is more stable.

Note that this algorithm is parallelizable because different $n_{l}$ can be sent to different nodes to speed  
the process of estimating $\hat{\xi}\left(\cdot \right)$ for all $n_{l}$.
After obtaining $\hat{\xi}(n_{l})$ for each value of $n_{l}$, we estimate $d_{VC}$ by minimizing the squared distance between $\hat{\xi}(n_{l})$ and $\Phi_{d_{VC}}^{*}(n_{l})$. Our objective function is
\begin{eqnarray}
f_{n_{l}}(d_{VC}) = \sum_{l=1}^{L}\left(\hat{\xi}(n_{l}) - \hat{c}\sqrt{\frac{d_{VC}}{n_{l}}\log\left(\frac{2n_{l}e}{d_{VC}}\right)}\right)^{2},
\label{objective}
\end{eqnarray}
where $L$ is the number of design points.
Optimizing \eqref{objective} usually only leads to numerical solutions and in our work below, we set $b_{1}=b_{2}=W$ for convenience.
\section{Proof of Consistency}
\label{ProofOfConsistency}

In this section, we provide a proof of consistency for the estimator $\hat{d}_{VC}$ for $d_{VC}$ that we presented in Subsec. \ref{estimator}. In many respects, the structure of this proof should be credited to \cite{McDonald:etal:2011}. Our contribution is to adapt \cite{McDonald:etal:2011} to our stable estimator for the regression context. We begin with some notation and definitions.

Let $\Phi = \left\{ \phi_{d_{VC},c} : H\times I \rightarrow \mathbb{R}\right\}$ where $d_{VC} \in H = (1,M]$ for some large $M \in \mathbb{N}, c\in I = [a,b]\subset \mathbb{R}$, and 
\begin{eqnarray}
\label{phidvc}
\phi_{d_{vc},c}\left(n_{l}\right) = c\sqrt{\frac{d_{VC}}{n_{l}}\log\left(\frac{2n_{l}e}{d_{VC}}\right)}
\end{eqnarray}
as derived in Subsec. \ref{estimator} (see expression \eqref{UPPTHEO6}). In expression \eqref{phidvc}, we assume $L$ values $n_{1}, \ldots, n_{L}$ have been pre-specified. Fix a value of $c$ and let $\phi_{c}$ be the corresponding elements of $\Phi$. The proof holds for each fixed $c$ and if we optimize over $c$ to obtain $\hat{c}$ 
as explained in Subsec. \ref{estimator}, the convergence of $\hat{d}_{VC}$ to the true value $d_{VC}$ will only be 
faster.

Without loss of generality, we assume that because $\Phi$ is compact we can choose $R > \sup_{d_{VC}} \norm{\phi_{d_{VC}}}_L$ where the norm $\norm{\cdot}_L$ is derived from the inner product
$$
\langle f, g \rangle = \frac{1}{L}\sum_{l=1}^{L}f\left(n_{l}\right)g\left(n_{l}\right)
$$
for real valued functions of a real variable. 
Thus $\phi_{d_{VC}} = \left(\phi_{d_{VC}}\left(n_{1}\right), \ldots, \phi_{d_{VC}}\left(n_{L}\right)\right)$ (where the subscript $c$ on the $\phi_{d_{VC},c}\left(n_{L}\right)'s$ in 
expression \eqref{phidvc} have been dropped for ease of notation). Now we can consider the class
\begin{eqnarray}
\label{Rbound}
\Phi_{c}(R) = \left\{\phi \in \Phi_{c} : \norm{\phi - \phi_{d_{VC}}}_{Q} < R\right\},
\end{eqnarray}
where $\phi_{d_{VC}}$ is the element of $\Phi_{c}$ corresponding to the correct value of $d_{VC}$. 
For a given $n_{l}$, we have 
\begin{eqnarray}
\hat{\xi}(n_{l})= \frac{1}{b_{2}}\sum_{i=1}^{b_{2}}r_{b_1,i}(n_{l})
\end{eqnarray}
where $r_{b_1,i}(n_{l})$ is the $i$ bootstrapped value of the integrand of $\Delta_{m}$ for each $n_{l}$, $i = 1,\ldots, W$ and $l = 1, \ldots, L$. In vector form, write $\hat{\xi} = \left(\hat{\xi}\left(n_{1}\right), \ldots, \hat{\xi}\left(n_{L}\right)\right)$.
Using \eqref{regmodel}, each $\hat{\xi}\left(n_{l}\right)$ can be represented as 
\begin{eqnarray}
\hat{\xi}\left(n_{l}\right) = \phi_{d_{VC}}\left(n_{l}\right) + \epsilon\left(n_{l}\right).
\end{eqnarray}
We have the following result.

\begin{theorem}
\label{TheoConsis}
  Suppose the true $d_{VC} \in (0,M]$ and that $\forall i = 1, \ldots, W$, $\forall l = 1, \ldots, L$, $r_{b_{1},i}\left(n_{l}\right) \sim N\left(\phi_{d_{VC}}(n_l), \sigma^2\right)$ and independent, $E\left(\epsilon(n_l)\right) = 0$, $Var(\epsilon(n_l)) = \sigma^2$. Then, on $\Phi_c(R)$, as $n \rightarrow \infty$, $m \rightarrow \infty$ and $W = W(n) \rightarrow \infty$ at suitable rates we have that 
  \begin{eqnarray}
  \label{deuxprime}
  P\left(\norm{\phi_{\hat{d}_{VC}} - \phi_{d_{VC}}}_{L} \geq \delta\right) = {\cal{O}}\left(\frac{1}{W}\right).
  \end{eqnarray}
\end{theorem}
{\bf Remark:} 
In fact, the $r_{b_{1},i}(n_l)$'s are only approximately independent $N\left(\phi_{d_{VC}}(n_l), \sigma^2\right)$. 
However, as $n$ increases they become closer and closer to being independent $N\left(\phi_{d_{VC}}(n_l), \sigma^2\right)$, assuming $\phi_{d_{VC}}(n_l)$ is a tight enough upper bound, as $n, m \rightarrow \infty$ 
at appropriate rates.    Also, it is seen that if $L=L(n)$ is increasing then $\| \cdot \|_L$ averages the evaluations
of of more and more components of, say, $\phi_{\hat{d}_{VC}}$.  In the limit, this can be exihibited as an
integral, i.e. as a quadratic norm.  So, $\| \cdot \|_L$ can be regarded as an approximation of a
$L^2$-space norm that strengthens as a norm (or inner product)
as $n \rightarrow \infty$.  In Theorem \ref{TheoConsis}, if we controlled the distance between
$\| \cdot \|_L$  and its limit, we would get a stronger mode for consistency.   

\begin{proof}
By definition of $\phi_{d_{VC}}$, we have 
\begin{eqnarray}
\label{eqmerlin}
\sum_{l = 1}^{L}\left(\hat{\xi}\left(n_{l}\right) - c\sqrt{\frac{\hat{d}_{VC}}{n_{l}}\log\left(\frac{2n_{l}e}{\hat{d}_{VC}}\right)}\right)^2\leq \sum_{l = 1}^{L}\left(\hat{\xi}\left(n_{l}\right) - c\sqrt{\frac{d_{VC}}{n_{l}}\log\left(\frac{2n_{l}e}{\hat{d}_{VC}}\right)}\right)^{2}
\end{eqnarray}
or more compactly $\norm{\hat{\xi} - \phi_{\hat{d}_{VC}}}_{L}^2 \leq \norm{\hat{\xi} - \phi_{d_{VC}}}_{L}^2$. Expanding both sides of \eqref{eqmerlin} gives
\begin{eqnarray}
\sum_{l =1}^{L}\left(\phi^{2}_{\hat{d}_{VC}}\left(n_{l}\right) - \phi_{d_{VC}}^{2}\left(n_{l}\right) \right) \leq 2\sum_{l=1}^{L}\hat{\xi}\left(n_{l}\right)\left(\phi_{\hat{d}_{VC}} -\phi_{d_{VC}}\left(n_{l}\right)\right) \nonumber
\end{eqnarray}
and hence 
\begin{eqnarray}
\norm{\phi_{\hat{d}_{VC}}}^{2}_{L} - \norm{\phi_{d_{VC}}}^{2}_{L}
&\leq&
\frac{2}{L}\sum_{l=1}^{L}\left(\phi_{d_{VC}}\left(n_{l}\right) + \epsilon\left(n_{l}\right)\right)\left(\phi_{\hat{d}_{VC}}\left(n_{l}\right) - \phi_{d_{VC}}\left(n_{l}\right)\right)\nonumber \\
&=&
\frac{2}{L}\sum_{l=1}^{L}\left(\phi_{d_{VC}}\left(n_{l}\right)\phi_{\hat{d}_{VC}}\left(n_{l}\right) -\phi^{2}_{d_{VC}}\left(n_{l}\right) \right.   \nonumber \\
 && \quad \quad \quad +
 \left. \epsilon\left(n_{l}\right)\left(\phi_{\hat{d}_{VC}}\left(n_{l}\right) - \phi_{d_{VC}}\left(n_{l}\right)\right)\right) \nonumber
\end{eqnarray}
Rearranging gives
\begin{eqnarray}
\norm{\phi_{\hat{d}_{VC}}}_{L}^{2} - 2\langle\epsilon, \phi_{\hat{d}_{VC}}\rangle + \norm{\phi_{{d}_{VC}}}^{2}_{L} \leq 2\langle \epsilon, \phi_{d_{VC}} - \phi_{\hat{d}_{VC}}\rangle_{Q}\nonumber
\end{eqnarray}
where $\epsilon = \left(\epsilon\left(n_{1}\right), \ldots, \epsilon\left(n_{L}\right)\right)$, i.e.
\begin{eqnarray}
\norm{\phi_{\hat{d}_{VC}} - \phi_{d_{VC}}}^{2}_{L} \leq 2\langle \epsilon, \phi_{\hat{d}_{VC}} - \phi_{d_{VC}}\rangle_{Q}.
\end{eqnarray}
It is seen that the LHS is the main quantity we want to control. We have
\begin{eqnarray}
\label{three}
P\left(\norm{\phi_{\hat{d}_{VC}}-\phi_{d_{VC}}}_{L} > \delta \right)
&\leq&
P\left(\langle \epsilon, \phi_{d_{VC}} - \phi_{\hat{d}_{VC}}\rangle \ge \frac{\delta^2}{2}\right)
 \nonumber \\
&\leq& P\left(\norm{\epsilon}_{L}\norm{\phi_{d_{VC}} - \phi_{\hat{d}_{VC}}}_L >\frac{\delta^2}{2}\right)
 \nonumber  \\
&\leq &\frac{2R^2}{\delta^2} E\norm{\epsilon}^{2}_{L}
\end{eqnarray}
using the Cauchy-Schwarz inequality, the bound in \eqref{Rbound}, and Markov's inequality.   

By construction, we have that
\begin{eqnarray}
\label{five}
E\norm{\epsilon}^{2}_{L} &=& \frac{1}{L}\sum_{l=1}^{L}E\left(\epsilon^{2}\left(n_l\right)\right)\nonumber\\
&=& \frac{1}{L}\sum_{l=1}^{L}E\left[\left(\frac{1}{W}\sum_{i=1}^{W}r_{b_{1},i}(n_l)\right) - \phi(n_{l})\right]^2\nonumber\\
&=& \frac{1}{L}\sum_{l=1}^{L}E\left[\frac{1}{W}\sum_{i=1}^{W}\left(r_{b_{1},i}(n_l)-\phi(n_l)\right)\right]^2\nonumber\\
&=&\frac{1}{LW^2}\sum_{l=1}^{L}\sum_{i=1}^{W}E\left(r_{b_{1},1}(n_l) - \phi(n_l)\right)^2\nonumber\\
&=&\frac{1}{LW^2}\sum_{l=1}^{L}\sum_{i=1}^{W}Var\left(r_{b_{1},1}(n_l)\right)\nonumber\\
&=& \frac{\sigma^2}{W}
\end{eqnarray}
Using \eqref{five} in \eqref{three} gives
\begin{eqnarray}
P\left(\norm{\phi_{\hat{d}_{VC}} - \phi_{d_{VC}}}_L\ge \delta\right) \leq \frac{2R^2 \sigma^2}{\delta^2 W} 
\end{eqnarray}
in which the upper bound decreases as $n$ increases because $W(n)$ is
increasing, thereby giving \eqref{deuxprime}.
\end{proof}
Note this proof allows $S$ to increase provided the rate of increase is slow enough 
with $n$ i.e., there exist sequences so that we can set $S = S(n)$ and still retain \eqref{deuxprime}.

A notable difference between \eqref{deuxprime} and the corresponding theorem in \cite{McDonald:etal:2011} is that our simplified result effectively only gives
\begin{eqnarray}
P\left(\norm{\phi_{\hat{d}_{VC}} - \phi_{d_{VC}}}_{L}\ge \delta\right) = \mathcal{O}\left(\frac{1}{W}\right)
\end{eqnarray}
rather than $\mathcal{O}\left(e^{-\gamma W}\right)$ for some $\gamma>0$, a much faster rate. 
We conjecture that the more sophisticated techniques used in \cite{McDonald:etal:2011} could 
be adapted to our setting and thereby give an exponentially fast rate of convergence of 
$\hat{d}_{VC}$ to $d_{VC}$ in probability. However, as yet, we have not been 
able to show this. Also, although it is suppressed in the notation, our result implicitly 
requires $m\rightarrow \infty$ to justify the use of $\phi_{d_{VC}}$.

Using Theorem \ref{TheoConsis}, we can show that our $\hat{d}_{VC}$ is consistent. Suppose that $\phi_{d_{VC}}(\cdot)$ is Lipschitz i.e. $\forall n_{l}, \exists \kappa = \kappa\left(n_{l}\right)$ so that $\kappa(n)\left|d_{VC}-d_{VC}^{'}\right|\leq \left|\phi_{d_{VC}}\left(n_{l}\right)-\phi_{d_{VC}^{'}}\left(n_{l}\right)\right|$, where $\kappa(n)$ is bounded on compact sets. Since the form of $\phi_{d_{VC}}\left(n_{l}\right)$ is known from \eqref{UPPTHEO6}, it is clear that the uniform Lipschitz condition we have assumed actually holds at least for appropriately chosen compact sets. We also observe that for $c\in I$ there exists a neighborhood $B\left(c,\epsilon_{l}\right), \eta>0$, on which \eqref{deuxprime} is true. Cover $I\times \mathbb{H}$ by sets of the form   $B\left(c,\eta\right)\times \left\{d_{VC}\right\}$; finitely many will be enough since $I\times \mathbb{H}$ is compact.

\begin{theorem}
\label{Theo7}
  Given that the assumptions of Theorem \ref{TheoConsis} hold and that $\phi_{d_{VC}}(\cdot)$ is Lipschitz, we have, as $n \rightarrow \infty$, that
  \begin{equation}\label{consistencyresult}
    P\left(\left|\hat{d_{VC}} - d_{VC}\right|\geq \delta\right)
 \leq \frac{2 R^2 \sigma^2}{\delta^2\kappa W}\sum_{s=1}^{S}2^s = {\cal{O}}\left(\frac{1}{W}\right),~
as ~ W_n \rightarrow \infty,
  \end{equation}
  where $\kappa = \sqrt{\frac{1}{L}\sum_{l=1}^{L}\kappa\left(n_{l}\right)}$ is the overall Lipshitz constant.
\end{theorem}

\begin{proof}
  Since all $L$ of the $\phi_{d_{VC}}\left(n_{l}\right)$'s are at least locally Lipschitz, their local Lipschitz inequlities can be summarized by an inequality of the form 
  \begin{equation}\label{blip}
    \begin{split}
       \left|d_{VC}-d_{VC}'\right|\sqrt{\frac{1}{L}\sum_{l=1}^{L}\kappa\left(n_{l}\right)} & \leq \sqrt{\frac{1}{L}\sum_{l=1}^{L}\left(\phi_{d_{VC}}\left(n_{l}\right)-\phi_{d_{VC}^{'}}\left(n_{l}\right)\right)^2} \\
         & =\|\phi_{d_{VC}}\left(n_{l}\right)-\phi_{d_{VC}^{'}}\left(n_{l}\right)\|_L,
    \end{split}
  \end{equation}
  where $d_{VC}$ is the true value and $d^{'}_{VC}$ is any other value in $\mathbb{H}$, and any extra constant from the local Lipschitz factors are assumed to have been absorbed into the $\kappa\left(n_{l}\right)$'s as needed.
  Let $\kappa = \sqrt{\frac{1}{L}\sum_{l=1}^{L}\kappa\left(n_{l}\right)}$. Using Theorem \ref{TheoConsis}, and \eqref{blip} we have
  \begin{equation}\label{ccprofend}
    \begin{split}
       P\left(\left|\hat{d}_{VC}-d_{VC}\right|\geq \delta\right) & \leq P\left[\|\phi_{\hat{d}_{VC}}\left(n_{l}\right)-\phi_{d_{VC}}\left(n_{l}\right)\|_{L}\geq \delta \kappa\right] \\
         & \leq \frac{2R^2\sigma^2}{\kappa \delta^2 W}\sum_{s=1}^{S}2^{s},
    \end{split}
  \end{equation}
where the last upper bound decreases as $W=W_n \rightarrow\infty$ as $n \rightarrow \infty$,  giving \eqref{consistencyresult}.
\end{proof}

\section{Simulation Studies}
\label{chap:Numerical:Studies}

For any model, we can evaluate the LHS of  \eqref{objfn} from Theorem \ref{Theo23and4} by Algorithm $\# 1$ in Sec. \ref{Algo1}. Then, we can use nonlinear regression in \eqref{objective} to find $\hat{d}_{VC}$. So, it is seen that $\hat{d}_{VC}$ is a function of the conjectured model. In principle, for any given model class, the VC dimension can be found, so our method can be applied.

Since our goal is to estimate the true VC dimension, when a conjectured model $P\left(\cdot\mid\beta\right)$ is linear and correct, we expect $VC\left(P\left(\cdot\mid\beta\right)\right)\cong\hat{d}_{VC}$. By the same logic, if $P\left(\cdot\mid\beta\right)$ is far from the true model, we expect $VC\left(P\left(\cdot\mid\beta\right)\right)\gg
\hat{d}_{VC}$ or $VC\left(P\left(\cdot\mid\beta\right)\right)\ll
\hat{d}_{VC}$. This suggests we estimate $d_{VC}$ by seeking
\begin{equation}
\label{Consis}
\hat{d}_{VC} = \arg\min_{k}\left|VC\left(P_{k}\left(\cdot\mid\beta\right)\right)- \hat{d}_{VC, k}\right|\leq t,
\end{equation}
where $\left\{P_{k}\left(\cdot\mid\beta\right)|k = 1,2,\cdots, K\right\}$ is some set of models and $\hat{d}_{VC, k}$ is calculated using model $k$, $t$ is a positive and usually small number that such that $t\leq 2$.
In the case of linear model, with $q = 1,2,\cdots, Q$ explanatory variables, we get
\begin{equation}
\label{eqcon}
\hat{d}_{VC} = \arg\min_{q}\left|q - \hat{d}_{VC, q}\right| \leq t
\end{equation}
where $\hat{d}_{VC, q}$ is the estimated VC dimension for model of size $q$.
Note that \eqref{Consis} can identify a good model even when consistency fails. The reason is that \eqref{Consis} only requires a minimum at the VC dimension not convergence to the true VC dimension which may be any model under consideration.
Here, we use a variation on \eqref{eqcon} by choosing the smallest local minimum of $\left|q-\hat{d}_{VC, q}\right|$, effectively setting $t=0$.
In practice, this does not always give a unique model and this may not be desireable.   On the other hand, 
choosing $t>0$ can give a collection of models, which may be desireable in some settings.

Our simulations are based on linear models, since for these we know the VC dimension equals the number of parameters in the model, see \cite{anthony2009neural}. To establish notation, we write the regression function as a linear combination of the covariates $X_{j}$, $j=0,1,\cdots, p$,
$$
y = f(x,\beta) = \beta_{0} + \beta_{1}x_{1}+\beta_{2}x_{2}+\cdots+\beta_{p}x_{p}=\sum_{j=0}^{p}\beta_{j}x_{j}.
$$
Given a dataset, $\left\{\left(x_{i},y_{i}\right),~ i =1,2,\ldots,n\right\}$, the matrix representation is

$$
Y = X\beta + \epsilon 
$$
where $Y$ is the $n\times1$ vector of response values, $X$ is the $n\times (1+p)$ matrix with rows $\left(1, x_{1,1}, x_{1,2},\ldots, x_{1,p}\right)$, $\beta = \left(\beta_{0},\beta_{1}, \ldots, \beta_{p}\right)^T$
is the vector of model parameters, and $\epsilon = \left(\epsilon_1, \epsilon_2,\ldots,\epsilon_n\right)^T$ is a $n\times 1$ mean zero Gaussian random vector.
Now, the least squares estimator $\widehat{\beta}$ is given by
$
\widehat{\beta} = \left(X^{'}X\right)^{-1}X^{'}Y.
$

Our simulated data is analogous. We write
$$
Y = \beta_{0}x_{0} + \beta_{1}x_{1} + \beta_{2}x_{2} + \cdots + \beta_{p}x_{p} + \epsilon
$$
$$
\quad \hbox{where} ~ \epsilon ~ \sim N(0,\sigma_{\epsilon}=0.4)\quad x_{0} = 1, ~ \beta_{j}\sim N(\mu = 5, \sigma_{\beta} = 3), ~ \hbox{for $j=1,2,\cdots p$,}
$$
$$
 \quad x_{j} \sim N(\mu=5, \sigma_{x} = 2),
$$ 
in which all $\beta$'s, $x$'s and $\epsilon$'s are independently generated.
\noindent
We center and scale all our variables, including the response. Initially, we use a nested sequence of model lists. If our covariates were highly correlated, before applying our method we could de-correlate them by sphering, i.e. transforming
the covariates using their covariance matrix so they become approximately uncorrelated with variance one, see \cite{Murphy2012} p. 144.


\subsection{Analysis of Synthetic Data}
\label{Simulation}
In Subsec. \ref{PropoWorks}, we begin by presenting simulation results to verify our estimator for VC dimension is consistent for the VC dimension of the true model. Of course, since our results are only simulations, we do not always get perfect consistency; sometimes our $\hat{d}_{VC}$ is off by one. (As we note later, this can often be corrected if the sample size is larger.)
In Subsec. \ref{Dependency}, we will look at simulations where results do not initially appear to be consistent with the theory. However, we show that for larger values of $p$, larger values of $n$ are needed. Also, as $p$ increases, we must choose $n_{l}$'s that are properly spread out over $[0,n]$.
and these $n_{l}$'s seem to matter less as $n\rightarrow \infty$. We suggest this is necessary because \eqref{objfn} is only an upper bound that conjecture tightens as $n$ increases. 

\subsubsection{Some first examples}
\label{PropoWorks}
In this subsection, we implement simulations for model sizes $p = 15, 30, 40, 50 ~ \hbox{and}~ 60$ and we present the results for all these cases for six model selection techniques AIC, BIC, CV $\widehat{ERM}_{1}$, $\widehat{ERM}_2$ and VC dimension (VCD).
For $p = 15, 30$, we use a sample size of $n=400$.  The design points are 
$N_{L} = \left\{50, 100, 150, 200, 250, 300, 400\right\}$; $m=10$; and the number of bootstrap samples is $W=50$. 
For $p = 40, 50 ~\hbox{and}~60$, we use a sample size of $n=600$.  The design points are 
$N_{L} = \left(75, 150, 225, 300, 375, 450, 525, 600\right)$; $m=10$; and the number of bootstrap 
samples is $W=50$. For these cases, we fit two sets of models; the first set uses a 
subset of our covariates to estimate the VC dimension, and in the second set, we added some `decoys' (their  
$\beta$'s in the generation of the responses are zeros).
Outputs of these simulations are given in Figs. \ref{erm15}-\ref{erm60}, in which
$\widehat{ERM}_{1}$, $\widehat{ERM}_{2}$, VCD, AIC, BIC, and CV,  are used to identify a good model.
\noindent
\begin{figure}
	\centering
    \includegraphics[width = \textwidth]{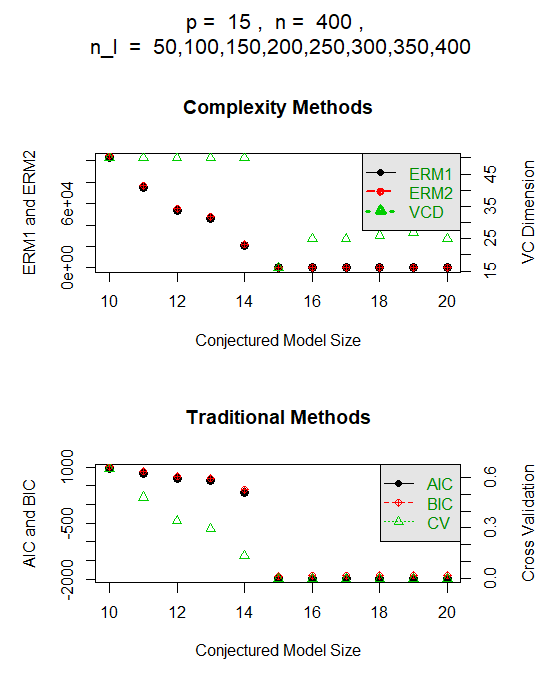}
    \caption{Upper: Values of $\widehat{ERM}_{1}$, $\widehat{ERM}_{2}$, and VC dimension. Lower: Values of AIC, BIC, and CV for $p=15$, $\sigma_{\epsilon}=0.4$, $\sigma_{\beta}=3$, $\sigma_{x} = 2$}
    \label{erm15}
\end{figure}

\begin{figure}
        \centering
        \includegraphics[width = \textwidth]{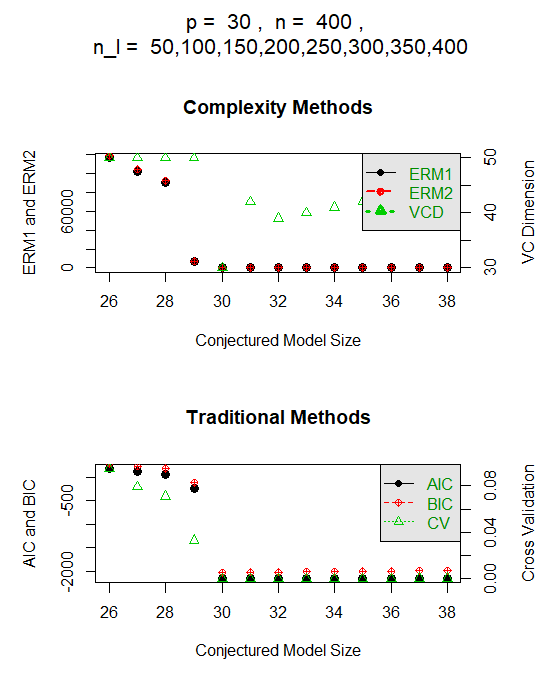}
    \caption{Upper: Values of $\widehat{ERM}_{1}$, $\widehat{ERM}_{2}$, and
    VC dimension (VCD). Lower: AIC, BIC, and CV values for $p=30$, $\sigma_{\epsilon}=0.4$, $\sigma_{\beta}=3$, $\sigma_{x} = 2$}
    \label{erm30}
\end{figure}

\begin{figure}
	\centering
    \includegraphics[width = \textwidth]{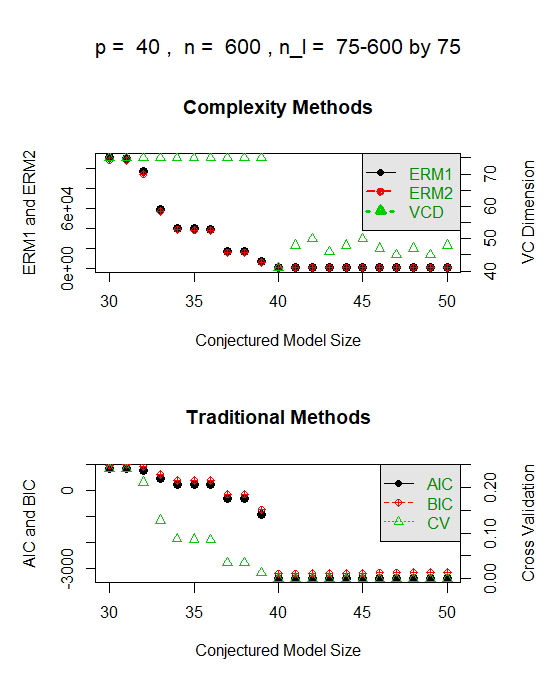}
    \caption{Upper: Values of $\widehat{ERM}_{1}$, $\widehat{ERM}_{2}$ and VC dimension (VCD). Lower: Values of AIC, BIC, and CV, for $p=40$, $\sigma_{\epsilon} = 0.4$, $\sigma_{\beta}=3$, $\sigma_{x} = 2$}
    \label{erm40}
\end{figure}

\noindent

\begin{figure}
	\centering
    \includegraphics[width = \textwidth]{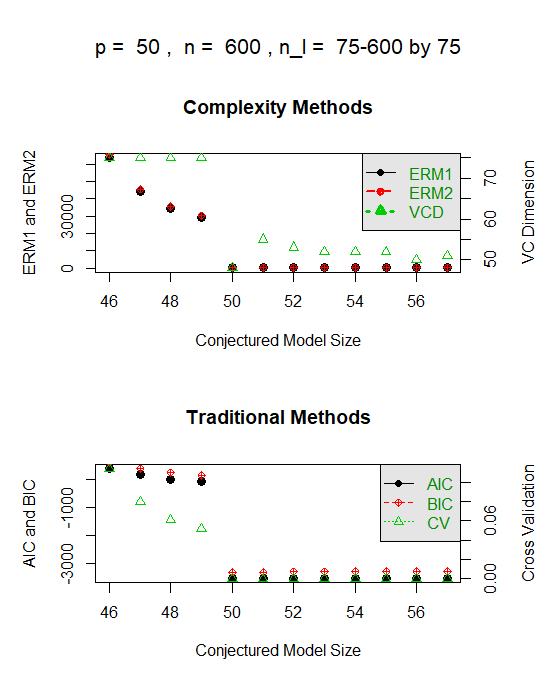}
    \caption{Upper: Values of  $\widehat{ERM}_{1}$, $\widehat{ERM}_{2}$, and VC dimension (VCD). Lower: Values of AIC, BIC, and CV for $p=50$, $\sigma_{\epsilon} = 0.4$, $\sigma_{\beta}=3$, $\sigma_{x} = 2$}
    \label{erm50}
\end{figure}

\begin{figure}
	\centering
    \includegraphics[width = \textwidth]{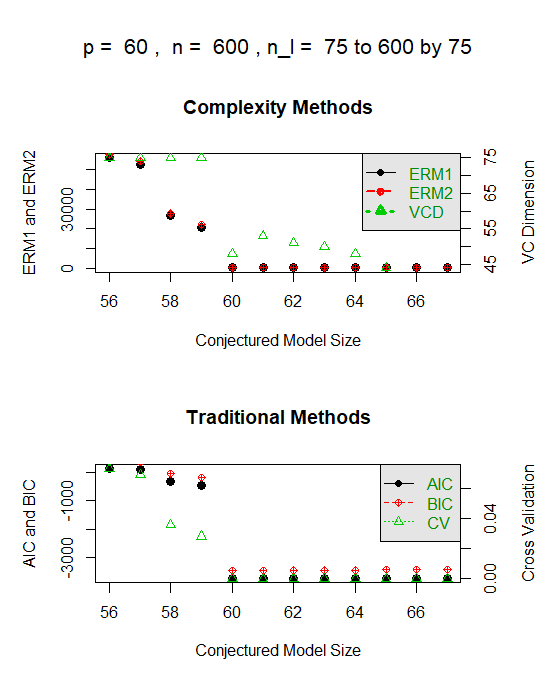}
    \caption{Upper: Values of $\widehat{ERM}_{1}$, $\widehat{ERM}_{2}$, and VC dimension (VCD). Lower: Values of AIC, BIC, and CV for $p=60$, $\sigma_{\epsilon} = 0.4$, $\sigma_{\beta}=3$, $\sigma_{x} = 2$}
    \label{erm60}
\end{figure}

By examining Figs \ref{erm15} to \ref{erm60},
we see that, for each given true model of pre-specified size, we fitted a list of nested models. 
When the size of the conjectured model is strictly less than that of the true model, the estimated VC dimension equals the minimum value of the design points, and the values of the AIC, BIC, CV, $\widehat{ERM}_{1}$, $\widehat{ERM}_{2}$  are high. These latter values typically decrease as the conjectured models become similar to the true model. For this range of model sizes, when the conjectured model exactly matches the true model, the estimated VC dimension ($\hat{d}_{VC}$) is closest to the true value. Except for Fig. \ref{erm60}, the biggest discrepancy (of size 2) occurs for $p=50$; by contrast, for every other case the difference between the true value and the estimated VC dimension is at most one.  Correspondingly, when the size of the conjectured of the true model is above the size of the true model, all six methods increase. The difference between our method (VCD) and $\widehat{ERM}_1$, $\widehat{ERM}_2$, AIC, BIC and CV is that the VCD values are much higher than the other values and tend to increase.  Thus, they more decisively
indicate which model should be taken as true.

The smallest discrepancy between the size $d_{VC}$ of the model and $\hat{d}_{VC}$ usually occurs at the true model. This indicates that $\hat{d}_{VC}$ is consistent.
In addition, even though the VCD values generally increase as the size of the conjectured 
model exceeds the size of the true model, in some cases, past a certain value $d_{VC}$, the VCD 
value may flatline as well. The problem with flatlining or even decreasing past a 
certain value of $d_{VC}$ occurs mostly when $n$ is not large enough relative to $p$.

In Figure \ref{erm60}, the behaviour of all six model selection methods is qualitatively the same, however, $\hat{d}_{VC}$ is far from the true value.
We suggest that this discrepancy occurs because the sample size is too small compared to $p$ and the choice of the design points is poor. Indeed, we suggest that in practice the effect of poorly chosen of design points is exacerbated when $n$ is not large enough. For instance, in Fig. \ref{ERM60700}, we see that $n = 700$ as opposed to $n = 600$,
 and very similar design points, the VCD performs better, rising when the conjectured models are too large.

\begin{table}[ht]
  \centering
\begin{tabular}{cccccccc}
   & Fig. \ref{erm15} & Fig. \ref{erm30}  & Fig. \ref{erm40} & Fig. \ref{erm50} & Fig. \ref{erm60} & Fig. \ref{ERM70} & Fig. \ref{ERM60700}\\
  \midrule
 $\frac{n}{p}$ & 27 & 13 & 15 &   12 & 10 & 10 & 12 \\
\end{tabular}
  \caption{Increase of the sample size relative to the size of the true model for Figs. \ref{erm15}--\ref{ERM60700}}\label{NoverP}
\end{table}
Table \ref{NoverP} gives the ratio of the sample size to the size of the model for Fig. \ref{erm15}--\ref{ERM60700}. 
Overall, we see that the higher $\frac{n}{p}$ is, the better the discrimination of $\hat{d}_{VC}$ over models is. The values of $\frac{n}{p}$ for Figs. \ref{ERM70} and \ref{ERM60700} merit some explanations. First as a  general rule, ones wants $\frac{n}{p} \geq 10$ for good parametric inference. Here, we are doing model selection as well as parametric inference, so we expect to require $\frac{n}{p}\geq 10$ for good inference. Examining Figs. \ref{erm60} and \ref{ERM70}, both of which have $\frac{n}{p} = 10$ we see that in the former inference poor while in the latter inference is adequate. We attribute this difference to the effect of having choosing better design points in Fig. \ref{ERM70}. Otherwise put, in some cases, intelligent choice of design points can compensate for insufficient sample size.

We argue that estimating VC dimension directly is better than using $\widehat{ERM}_{1}$ or $\widehat{ERM}_{2}$. There are several reasons. First, the computation of $\widehat{ERM}_{1}$ and $\widehat{ERM}_{2}$ requires $\hat{d}_{VC}$. It also requires a threshold $\eta$ be chosen (see Propositions \ref{subProp1} and \ref{subProp2}) and is more dependent on $m$ than $\hat{d}_{VC}$ is. Being more complicated than $\hat{d}_{VC}$, $\widehat{ERM}_{1}$, $\widehat{ERM}_{2}$ will break down faster than $\hat{d}_{VC}$. This is seen, for instance in tables of \cite{Merlin:etal:2017} Chap. 3 and the discussion there.
More generally, we argue that $\widehat{ERM}_{1}$, and $\widehat{ERM}_{2}$ break down faster than $\hat{d}_{VC}$ with increasing $p$, if the sample size is held constant. That is, $\widehat{ERM}_{1}$ and $\widehat{ERM}_{2}$ are less efficient than $\hat{d}_{VC}$. Results in \cite{Merlin:etal:2017} show that all of these conclusions are 
qualitatively the same if $\sigma_{\epsilon}$, $\sigma_{x}$ or $\sigma_{\beta}$ are varied.

Finally for this subsection, we reiterate our observation that in practice, when our VC dimension technique is used properly, it gives a well defined minimum as the estimator for $d_{VC}$. In particular, it is much more sensitive to over fit than the other five methods, thereby giving better sparsity of models.
 
 \subsubsection{Dependency on The Sample Size and Design Points}
\label{Dependency}

Our goal here is to show how we can improve the quality of our estimates by increasing $N$ and/or tuning the design points $n_l$. We started this in Sec. \ref{PropoWorks}, but now we want to provide some practical guidelines. We perform simulations using $N = 700$ and 2000 so as to get contrasting results. 
We use two values of $p$, 70 in Fig \ref{ERM70} and 60 in Figs. \ref{ERM60700}--\ref{ERM602000NL}.

First consider Figs. \ref{erm60} and \ref{ERM70}. Both show the results of using six model selection techniques. When the size of the conjectured model is strictly less than the size of the true model, $\hat{d}_{VC}$ is equal to the smallest design point. However, when the conjectured model exactly matches the true model, $\hat{d}_{VC} \approx 50, 61$, respectively, underestimating $d_{VC}$.
Interestingly, if we simply look at the minimal VCD values they occur at conjectured models of size 60 and 70 respectively, the true values of $p$.
When the conjectured model is more complex than the true model, the VCD value is visibly higher than the VCD value for the true model. Our observation for the other five methods are as before i.e. they are less affected by the small sample size, but loosely speaking decrease and flatline.

Figure \ref{ERM60700} gives the plots for the six model selection methods for for $N = 700$ and $p=60$. Here, $\hat{d}_{VC} = 57$. This value is closer to the true value than $\hat{d}_{VC} = 50$ in Fig. \ref{erm60}. There are no qualitative changes to the other five methods. So, small changes in the design points can have large numerical effect when $N$ is too small.

Figs. \ref{ERM602000}--\ref{ERM602000NL} have $p=60$, $n = 2000$; the difference from one to the next is only
in the design points. The non-VCD method methods are qualitatively the same as in earlier figures. However, $\hat{d}_{VC} = 59, 50$ and 56, respectively. Clearly, the best model selection occurs in Fig. \ref{ERM602000}, in which the design points cover the region $[0,n]$. The second best occurs in Fig. \ref{ERM602000NL}. The worst occurs in Fig. \ref{ERM602000NL2}. So, we tentatively suggest that using more design points over a small range is not as good as using fewer design points over a larger range.

We leave the question of optimally choosing the design points as future work even though we also suggest
design points matter less as $n$ increases and
\begin{enumerate}
  \item Good choices of design points are spread over $[0,n]$;
  \item More design points should be in $[\frac{n}{2},n]$ than in $[0,\frac{n}{2}]$, but neither should be empty;
  \item  Good, resp. poor, choices of design points tend to remain good, resp. poor, as $n \rightarrow \infty$.
\end{enumerate}

\begin{figure}
	\centering
    \includegraphics[width = \textwidth]{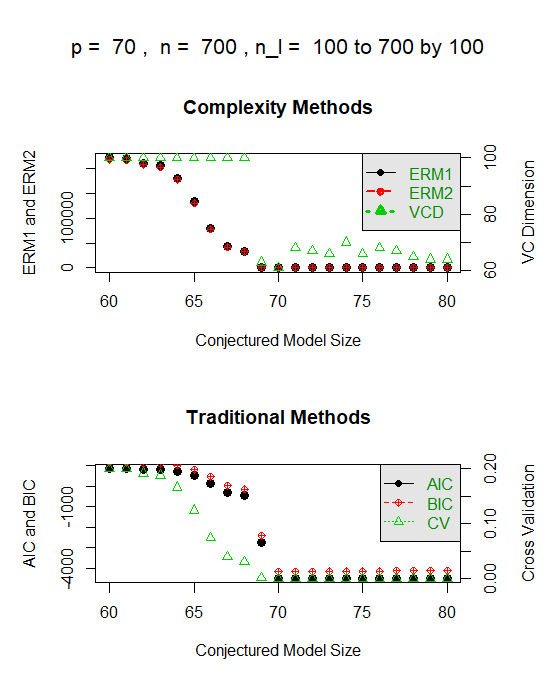}
    \caption{Upper: Values of $\widehat{ERM}_{1}$, $\widehat{ERM}_{2}$, and VC dimnsion. Lower: Values of AIC, BIC and CV  for $p=70$, $\sigma_{\epsilon} = 0.4$, $\sigma_{\beta}=3$, $\sigma_{x} = 2$}
        \label{ERM70}
\end{figure}

\begin{figure}
	\centering
    \includegraphics[width = \textwidth]{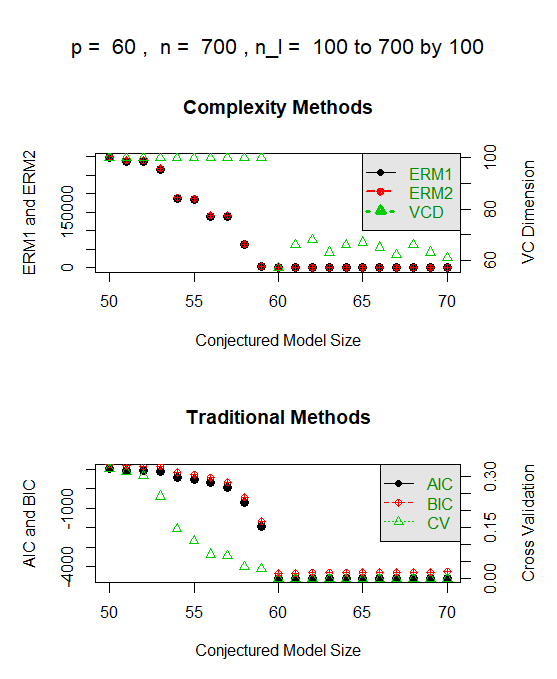}
    \caption{Upper: Values of $\widehat{ERM}_{1}$, $\widehat{ERM}_{2}$, and VC dimension. Lower: Values of AIC, BIC and CV  for $p=70$, $\sigma_{\epsilon} = 0.4$, $\sigma_{\beta}=3$, $\sigma_{x} = 2$}
        \label{ERM60700}
\end{figure}

\begin{figure}
	\centering
    \includegraphics[width = \textwidth]{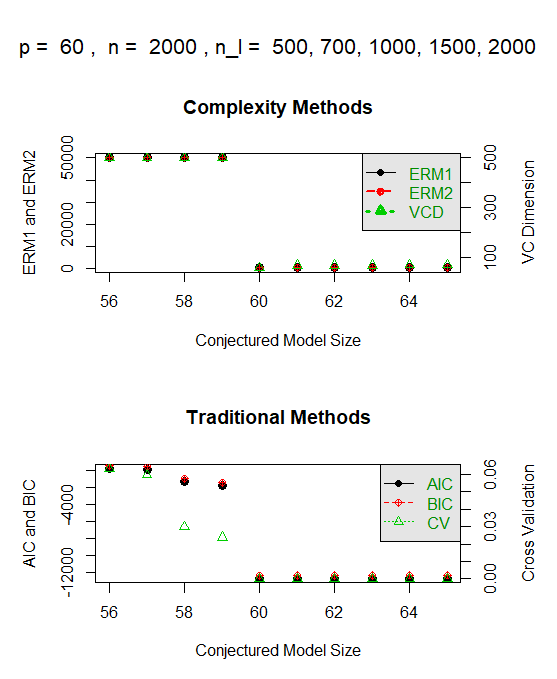}
        \caption{Upper: Values of $\widehat{ERM}_{1}$, $\widehat{ERM}_{2}$, and VC dimension. Lower: Values of  AIC, BIC and CV for $p=60$, $\sigma_{\epsilon} = 0.4$, $\sigma_{\beta}=3$, $\sigma_{x} = 2$}
        \label{ERM602000}
\end{figure}

\begin{figure}
	\centering
    \includegraphics[width = \textwidth]{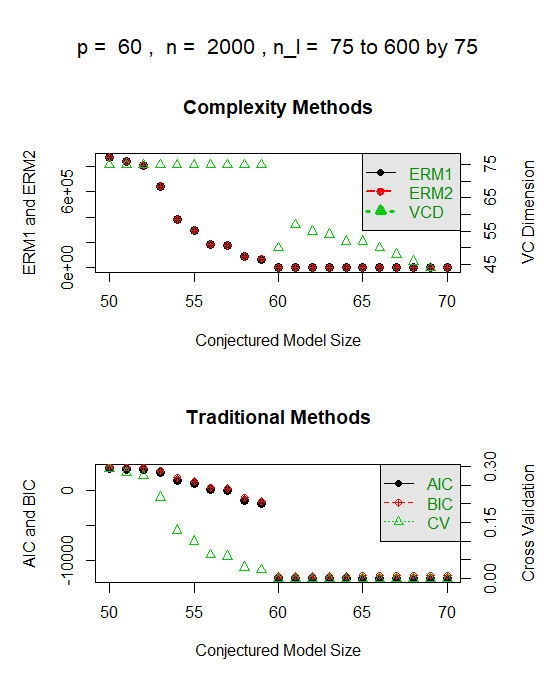}
        \caption{Upper: Values of $\widehat{ERM}_{1}$, $\widehat{ERM}_{2}$, and VC dimension. Lower: Values of  AIC, BIC and CV for $p=60$, $\sigma_{\epsilon} = 0.4$, $\sigma_{\beta}=3$, $\sigma_{x} = 2$}
        \label{ERM602000NL2}
\end{figure}

\begin{figure}
	\centering
    \includegraphics[width = \textwidth]{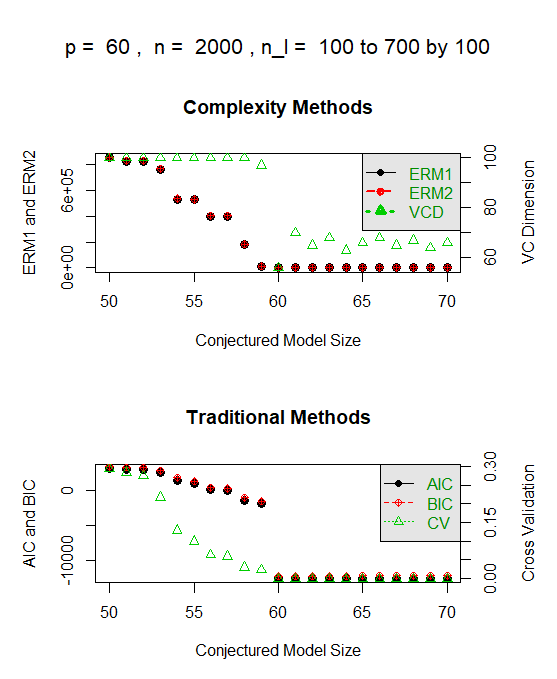}
        \caption{Upper: Values of $\widehat{ERM}_{1}$, $\widehat{ERM}_{2}$, and VC dimension. Lower: Values of AIC, BIC and CV for $p=60$, $\sigma_{\epsilon} = 0.4$, $\sigma_{\beta}=3$, $\sigma_{x} = 2$}
        \label{ERM602000NL}
\end{figure}

\section{Analysis of More Complex Data Sets}
\label{example}

The goal of this section is to evaluate our method on two benchmark datasets: {\sf {Tour de France}}\footnote{Tour de France Data was collected by Bertrand Clarke. More information can be found at  http://www.letour.fr/.} and {\sf{Abalone}} datasets. The analysis of the {\sf{Abalone}} dataset will be more extensive than that of the {\sf{Tour de France}} because its sample size is much larger, 4177 versus 103. Moreover, we can assume that the {\sf{Abalone}} dataset is independent across individual abalones.

We start this section by giving some information about our {\sf{Tour de France}} dataset in Sec. \ref{DescripStat}. Then, in Sec. \ref{AnalTour5} we analyze {\sf{Tour de France}} dataset using a model list based on $Year$ and $Distance$. This class is a sequence of models nested by SCAD.
 We evaluate our method by comparing $\hat{d}_{VC}$ to AIC, BIC, CV $\widehat{ERM}_{1}$ and $\widehat{ERM}_{2}$. In Subsec. \ref{AnalTourOutRemoved}, we look at the effect of outliers in the estimates $\hat{d}_{VC}$, $\widehat{ERM}_{1}$ and $\widehat{ERM}_{2}$. We present our analogous analysis of the {\sf{Abalone}} 
data set in Sec. \ref{Abalone}.

\subsection{{\sf{Tour de France Data}}}
\label{DescripStat}
The full dataset has $n=103$ data points. The data points are dependent (correlated) because many cyclists competed in the Tour for more than one year. Here we ignore the dependence structure because the correlations are small enough that the complication of accounting for them is not worthwhile. Each data point has a value of the response variable, the average speed in kilometers per hour (km/h) of the winner (Speed) of the Tour from 1903 to 2016. The explanatory variables are the Year (Y), the Distance (km) (D) of the Tour.  However, during World Wars 1 and 2 there was no {\sf{Tour de France}}, so we do not have data points for those years. We also see the effect of World War I on the speed of the winner of the tour:
The lowest speeds were in the years just after World War I, probably due to casualties. After World War II, there was also a decrease in average winning speed, but the decrease was less than that after World War I. 
Looking at Fig. \ref{TourdeFrancedata} we see a curvilinear relationship between Speed and Year. 
Although not shown here, there is a roughly linear relationship between Speed and Distance 
and that the variability of Speed increases with the Distance (D).

\begin{figure}[H]
	\centering
    \includegraphics[width = 15cm, height = 7cm]{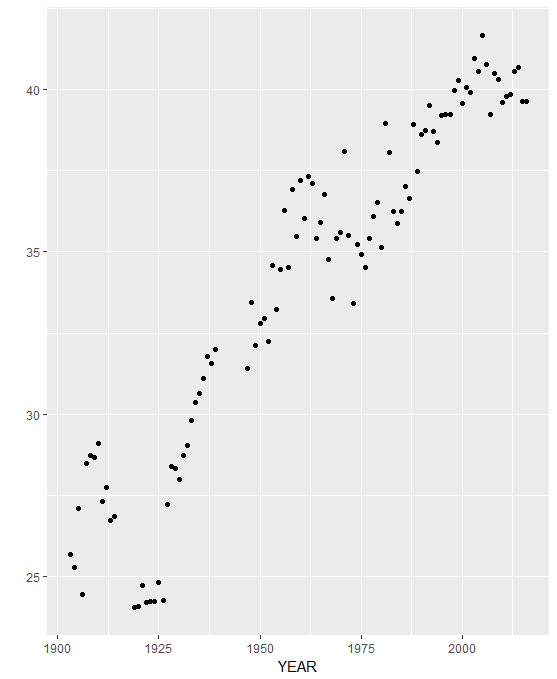}
        \caption{Scatter plot of winning Speed VS Year for the {\sf{Tour de France}} dataset. The cluster of points centered around Speed = 23 and Year = 1924 are potential outliers}
        \label{TourdeFrancedata}
\end{figure}

\subsubsection{Analysis of a nested collection of model lists of {\sf{Tour de France }}}
\label{AnalTour5}

We identify a nested model list using $Y$, $D$, $Y^2$, $D^2$ and the interaction between Year and distance denoted $Y:D$ as covariates. Because the size of the dataset is small, we can only use a small model list. We order the variables using the SCAD shrinkage method because it perturbs parameter estimates the least and satisfies an oracle property.

Under SCAD, the order of inclusion of variables is $Y$, $D$, $D^2$, $Y^2$, and $Y:D$. We therefore fit five different models. We use the six model selection techniques from Sec. \ref{Simulation} and include $\hat{d}_{VC}$ the original estimator in \cite{Vapnik:etal:1994} for the sake of comparison.
\begin{table}[ht]
  \centering
        \begin{tabular}{cccccccc}
   Model & $\hat{d}_{V}$ & $\hat{d}_{VC}$ & $\widehat{ERM}_{1}$  & $\widehat{ERM}_{2}$ & AIC & BIC & CV\\
  \midrule
  $Y$ & 20 & 4 & 16.42 & 44.95 & 84 & 79.67 & 0.1294 \\
    $Y,~D$ & 20  & 4& 15.10 & 42.83 &  77 & 71.66 & 0.1209 \\
    $Y,~D,~D^2$ & 20 & 4 & 11.21  & 36.37 & 24 & $\bf 26.21$  & 0.0727 \\
    $Y, D,~D^2,~Y^2$ & 20  & 4 & 11.09 & 36.16 & \bf{17} & 28.77 & \bf{0.0681}\\
    $Y, D,~D^2,~Y^2,~Y:D$ & 20 & 4 & $\bf 11.06$  &  $\bf 36.11$ & 19 & 32.96 & 0.0691 \\
\end{tabular}
  \caption{Model selection for the {\sf{Tour de France}} data using seven methods. The design points for $\hat{d}_{V}$ and $\hat{h}_{VC}$ are 20, 30, 40, 50, 60, 70, 80, 90, and 100 and $W =50$. So $\hat{d}_{V}$ equals the smallest design point in all cases. This problem frequently occurs for $\hat{d}_V$; it can also occur 
for $\hat{d}_{VC}$ but is rarely a severe problem.}\label{NestTheo5Direct}
\end{table}
 It is seen that \cite{Vapnik:etal:1994}'s original method is helpful only if it is reasonable to surmise that there are $15$ missing variables whereas our method uniquely identifies one of the models on the list. Even though there is likely no model for {\sf{Tour de France}} dataset that is accurate to infinite precision, our method is giving a useful result. Indeed, our method is choosing the fourth model list the same as indicated by AIC and CV. The BIC drops the $Y^2$ term which is not unreasonable because, as seen in Fig. \ref{TourdeFrancedata}, the curvilinearity is less than quadratic. $\widehat{ERM}_{1}$ and $\widehat{ERM}_{2}$ include the interaction term (which can be seen to be zero by a simple $t$-test). It may also be the case that the derivation of the BIC rests heavily on using independent data 
which is not the case here.

There is nothing a priori wrong with $\widehat{ERM}_{1}$ and $\widehat{ERM}_{2}$, but a smaller model (of size 4 using $\hat{d}_{VC}$) is preferred when justifiable. Alternatively, we may regard the difference among the $3^{th}$, $4^{th}$ and $5^{th}$ model lists as trivial for $\widehat{ERM}_{1}$, $\widehat{ERM}_{2}$, so they effectively lead to the model with $Y$, $D$ and $D^2$ as terms, since $\widehat{ERM}_{1}$ and $\widehat{ERM}_{2}$ both have a large decrease from the $2$ term to the 3 term model. That is, they give the same result as BIC which we think is inferior to the model chosen by $\hat{d}_{VC}$ which tries to capture the curvilinearity in Speed as a function of Year.

\subsubsection{ Analysis of The {\sf{Tour de France}} dataset with outliers removed}
\label{AnalTourOutRemoved}

The observations just after World War I may be outliers, because so many French men were 
killed during the war. So, we consider the data set formed by deleting the points from 1919 to 1926. 
Let us see how the six model selection techniques now behave.

The process of analyzing this reduced dataset is the same: We identify the nested model lists by SCAD and then find the models corresponding to $\hat{d}_{VC}$, AIC, BIC, CV, $\widehat{ERM}_{1}$, and $\widehat{ERM}_{2}$.  The results are given in Table \ref{Set2YDNest}.

\begin{table}[ht]
  \centering
  \begin{tabular}{ccccccc}
  Model Size &$\widehat{d}_{VC}$ & $\widehat{ERM}_{1} $ & $\widehat{ERM}_{2}$ & AIC & BIC & CV\\
  \midrule
    $Y$ & 4 & 12.87 & 40.72 & 67 & 75 & 0.1181\\
    $Y$, $D^2$ & 4 & 12.01 & 39.26 & 69 & 79 & 0.1336\\
    $Y$, $D^2$, $D$ & 4 & 11.66 & 38.36 & 46 & 59 & 0.0919 \\
    $Y$, $D^2$, $D$, $Y^2$ & \bf{4} & 11.48 & 38.34 & \bf{28} & \bf{43} & 0.0742\\
    $Y$, $D^2$, $D$, $Y^2$, $Y:D$ & 4 & \bf{11.35} &\bf{ 38.13} & 29 & 46 & \bf{0.0735} \\
\end{tabular}
\caption{Model selection for the {\sf{Tour de France}} dataset with outliers removed.}
  \label{Set2YDNest}
\end{table}

Under SCAD, the order of inclusion of our covariates is: $Y$, $D^2$, $D$, $Y^2$ and $Y:D$. This 
order is different from when we used all data points and shows that the outliers suggested $Y:D$ was
more important than it probably is.  Note also that when we used all the data points, $D$ was included before $D^2$ and $D^2$ was included after $Y^2$. With this new ordering we fit 5 different models.

From Table \ref{Set2YDNest}, if we choose a model using $\hat{d}_{VC}$, we get the same answer as in Sec. \ref{AnalTour5}, the model with four variables: $Y$, $D^2$, $D$, $Y^2$. AIC and BIC choose the same model 
probably because $(Y:D)$ has low correlation with Speed (-0.08). $\widehat{ERM}_{1}$, $\widehat{ERM}_{2}$ 
and CV choose the model of size 5, which we discount as before because $Y:D$ is only slightly correlated with 
Speed. That is, the reasoning in Subsec. \ref{AnalTour5} for why we think that the 
model chosen by $\hat{d}_{VC}$ is best continues to hold.

\subsection{Analysis of the {\sf{Abalone} Dataset}}
\label{Abalone}
The {\sf{Abalone}} dataset was first presented in \cite{Nash:etal:1994} and can be freely downloaded from
\url{http://archive.ics.uci.edu/ml/datasets/abalone}. {{\sf Abalone}} has been widely used in statistics and in machine learning as a benchmark dataset. It is known to be very difficult to analyze as either a classification or as a regression problem. Our goal in this section is to see how our method will perform and compare the result to other model selection techniques. Here, in addition to AIC, BIC, CV, $\widehat{ERM}_1$ and $\widehat{ERM}_2$, we add SCAD and ALASSO.

\subsubsection{Descriptive Analysis of {\sf{Abalone}} dataset}
\label{AbaDesc}
The $\sf{ Abalone}$ dataset  has 4177 observations with 8 covariates:  Length (mm) is the longest shell measurement, Diameter (mm), Height (mm) is the height measured with the meat, Whole weight (grams) is the whole weight of the abalone, Shucked weight (grams) is the weight of the meat, Viscera weight (grams) is the gut weight after bleeding, Shell weight is the shell weight after being dried. The response variable is Rings; the number of Rings is roughly the age of an abalone \footnote{There actually is an 8th covariate. Sex is a nominal variable with 3 categories: 
Male, Female and Infant.  We did not include it in our analysis because we wanted to limit attention to
continuous variables in this example.}.

\begin{figure}
    \centering
    \begin{subfigure}[b]{0.4\textwidth}
        \centering
        \includegraphics[width = \textwidth]{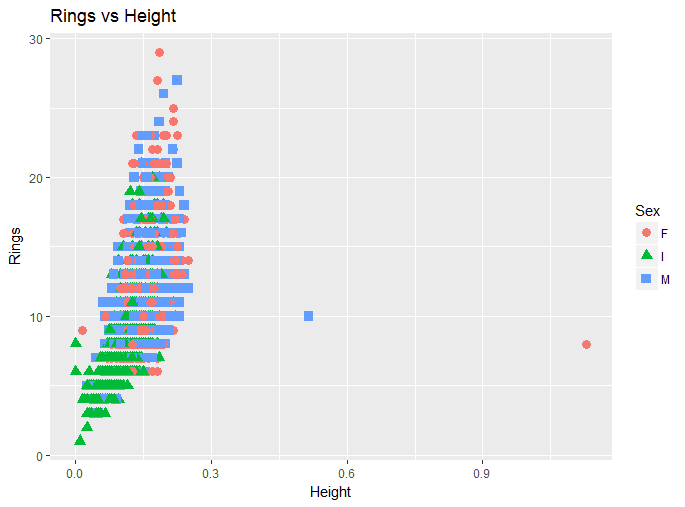}
        \caption{Scatter plot of Rings vs Height}
        \label{ScatterHeight}
    \end{subfigure}
    \begin{subfigure}[b]{0.4\textwidth}
        \centering
        \includegraphics[width = \textwidth]{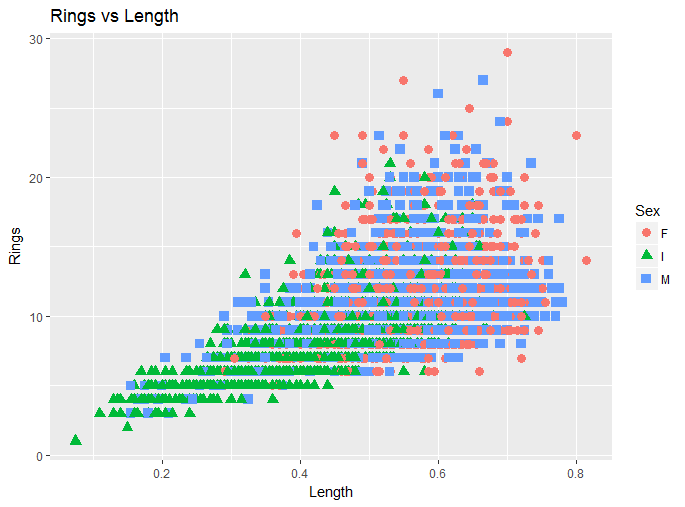}
        \caption{Scatter plot of Rings vs Length}
        \label{ScatterLength}
    \end{subfigure}
    \hfill
    \begin{subfigure}[b]{0.4\textwidth}
        \centering
        \includegraphics[width = \textwidth]{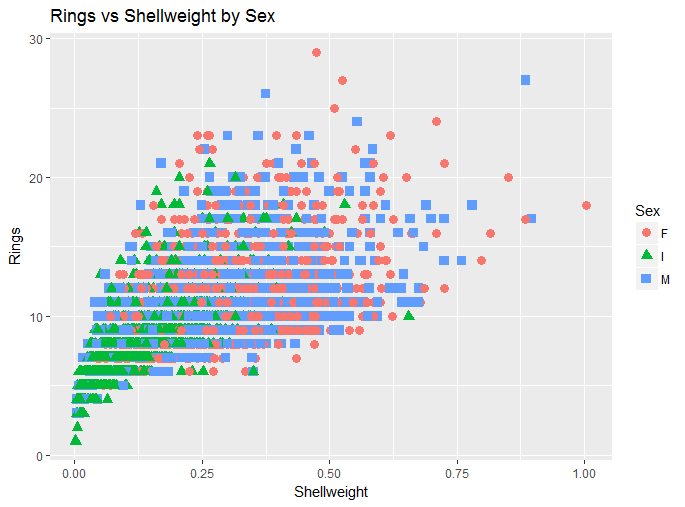}
        \caption{Scatter plot of Rings vs Shell weight}
        \label{ScatterShellweight}
    \end{subfigure}
    \hfill
    \begin{subfigure}[b]{0.4\textwidth}
        \centering
        \includegraphics[width = \textwidth]{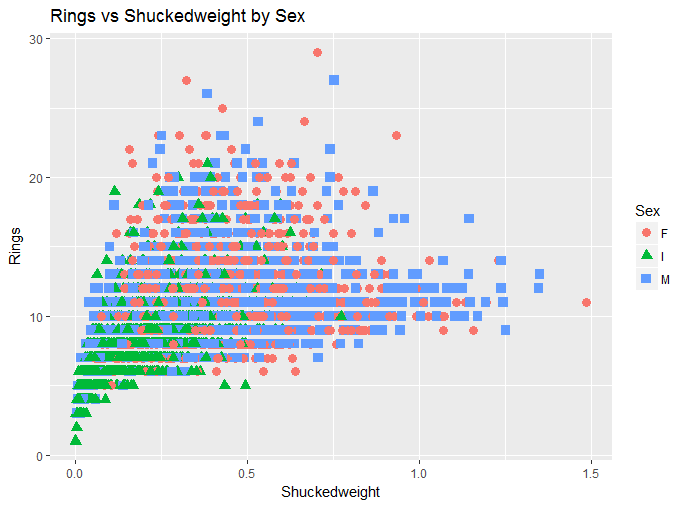}
        \caption{Scatter plot of Rings VS Shucked weight}
        \label{ScatterShuked}
    \end{subfigure}
    \caption{Scatter plot of Rings vs  Height,
    Length, Shell weight, Shucked weight by Sex}
    \label{ScatterPlot}
\end{figure}
Fig. \ref{ScatterPlot} shows pairwise scatterplots of Rings versus covariates. We see that no matter which 
covariates are chosen, the variability in the Rings increases as the covariates increase. We also observe that 
there is likely to be a curvilinear relationship between Rings and each of these covariates. However, the 
curvilinearity is not strong because of the variability. Indeed for Rings vs Length it is almost nonexistent. 
The three scatterplots we left out give near duplicate to the panels in Fig. \ref{ScatterPlot} due to colinearity. 
(Rings vs Diameter is nearly the same as Rings vs Length; and, Rings vs Viscera weight and Rings vs Whole 
weight is nearly the same as Rings vs Schucked weight.) Thus, as a simple analysis, we choose 7 linear 
terms in the 7 explanatory variables. 

\subsubsection{Statistical Analysis of the {\sf{Abalone}} data}

The model that we use to estimate the complexity of the response `Rings' is a linear combination of all the variables. To accomplish this, we first order the inclusion of variables in the model using correlation; see \cite{fan2008sure}. Under correlation between Rings and each of the explanatory variables, the order of inclusion of variables is as follows: Shell weight, Diameter, Height, Length, Whole weight, Viscera weight and Shucked weight. Using this ordering, we form a nested sequence of models and found $\hat{d}_{VC}$, AIC, BIC, CV, $\widehat{ERM}_1$ and $\widehat{ERM}_2$ for each model. The results are in Table \ref{AbaNested}.Below, we also compare our method to two sparsity methods namely SCAD and ALASSO.

From Table \ref{AbaNested}, we observe that $\hat{d}_{VC} = 8$ except for model of size 5. We regard $\hat{d}_{VC} = 9$ for a model of size 5 as a random fluctuation since it is close to 8 and our method, while stable, is not perfectly so. The full model is chosen using $\hat{d}_{VC}$ because it has the smallest distance between its size and $\hat{d}_{VC}$. However, the fact that model is first order of size 7 and $\hat{d}_{VC}=8$ suggests there may be a missing variable in the dataset, i.e., unavoidable bias. It is unclear without much further work whether including Sex as a dummy variable or including higher terms in the covariates will remedy this. We see some variability in the estimates of $\widehat{ERM}_{1}$ and $\widehat{ERM}_{2}$ as we include variables in the model. However, the key point is that there is a relatively big decrease in passing from model 6 to model 7. This observation is the same for AIC, BIC and CV. So, all methods effectively pick the model of size 7. The benefit of our model is that it provides evidence that a larger model is necessary.
\begin{table}[h]
  \centering
\begin{tabular}{ccccccc}
  Model Size & $\widehat{d}_{VC}$ & $\widehat{ERM}_{1} $ & $\widehat{ERM}_{2}$ & AIC & BIC & CV \\
  \midrule
  Shellweight & 8 & 2533 & 2597 & 9768 & 9787 & 0.6067\\
  Shellweight, Diameter & 8 & 2532 & 2596 & 9768 & 9793 & 0.6066\\
  Shellweight, Diameter, Height & 8 & 2507 & 2571 & 9729 & 9760 & 0.6111 \\
  Shellweight, Diameter, Height, Length & 8 & 2478 & 2542 & 9682 & 9720 & 0.6059\\
  5 & 9 & 2260 & 2324 & 9299 & 9343 & 0.5555\\
  6 & 8 & 2259 & 2320 & 9300 & 9351 & 0.5558\\
  7 & 8 & \bf{1974} & \bf{2031} & \bf{8738} & \bf{8795} & \bf{0.4849}  \\
\end{tabular}
\caption{Seven nested models under correlation with Rings for the {\sf{Abalone}} dataset and the results from six model selection techniques. For $\hat{d}_{VC}$ we set $L = 8$ and took the $n_l$'s to range from 1000 to 5000 in steps of 500.}
\label{AbaNested}
\end{table}

Next we turn to the results of a sparsity driven analysis. Since we are using seven explanatory variables and $n = 4177$, sparsity is not an important property for a model. Hence, we present these results for comparative purposes only.

First suppose SCAD is used as a model selection technique. The optimal value of $\lambda$ is found to be $\hat{\lambda} = 0.0027$. With this value of $\hat{\lambda}$, the best model must have 6 variables.
The variables that get into the model are, in order, Shell weight, Shucked weight, Height, Diameter, Viscera weight, and Whole weight. Thus, under SCAD, we are led to the model
    \begin{equation}
    \label{AbaEqua}
    \begin{split}
    \widehat{Rings} & = 0.36\cdot Diameter + 0.15\cdot Height + 1.40\cdot Wholeweight \\
     & - 1.39S\cdot Shuckedweight -0.34\cdot Visceraweight  \\
     & + 0.37\cdot Shellweight.
    \end{split}
    \end{equation}

    Analogous analysis under ALASSO leads us to the same six terms and the model is:

    \begin{equation} \label{AbaALASSO}
    \begin{split}
    \widehat{Rings} & = 3 + 11.62\cdot Diameter + 11.69\cdot Height + 9.21\cdot Wholeweight \\
     & - 20.24\cdot Shuckedweight -9.79\cdot Visceraweight  \\
     & + 8.63\cdot Shellweight.
    \end{split}
    \end{equation}
Both \eqref{AbaEqua} and \eqref{AbaALASSO} omit Length, and have the same terms, albeit with different coefficients. This may occur because there is a high correlation between covariates. SCAD and ALASSO, being sparsity methods, give one fewer term than the models in Table \ref{AbaNested}, although in Table \ref{AbaNested} Length is the fourth most important variable. This difference, too,  may reflect collinearity. 

Overall, it is seen that the largest model is the most reasonable and is identified by $\hat{d}_{VC}$ whose 
values also suggest without further analysis that even the full model is not large enough. 
This statement is also consistent with other analyses of this data.

\section{Application to a Typical Multi-type Agronomic Data Set }
\label{Application}

To demonstrate the use of our technique, we re-analyze the {\sf{Wheat}} data set presented and studied in
 \cite{campbell2003identification}, \cite{DILBIRLIGI200674}, and \cite{Dhungana:et:al:2007} from a 
non-complexity based standpoint. The $\sf{Wheat}$ data set has 2912 observation and 104 varieties. 
The experimental study was conducted in seven locations; Lincoln, NE,
in 1999 to 2001, and Mead and Sidney, NE, in 2000 and 2001. The design used in Lincoln, NE, in 
1999 was a Randomized Complete Block Design (RCBD) with four replicates. In other years, the 
design was an incomplete block design with four replicates where each replication consisted of 
eight incomplete blocks of thirteen entries.  The environments are diverse and representative of wheat 
producing areas of Nebraska. More information concerning the data set and the design structure 
can be found in \cite{campbell2003identification}.
The response variable is YIELD (MG/ha), the covariates that we used are 1000 kernel weight (TKWT), 
kernels per spike (KPS), Spikes per square meter (SPSM), height of the plant (HT), test weight 
(TSTWT(KG/hl)), and kernels per square meter (KPSM).
Often, in agronomic data sets, there are several classes of explanatory variables, here we have 
phenotypic, single
nucleotide polymorphisms (SNP's) and the variables defining the design. Through a series of 
examples, we show how to use our VC dimension based model selection procedure and verify 
that it gives good results compared with seven other methods.

Our collection of explanatory variables can be grouped into three categories: phenotype, SNP, and design variables. Treating these as groups of variables, we can, in principle fit $2^3 - 1 = 7$ different model classes, but not all these will make sense.  Specifically, we rule out models containing only SNP variables,  only design variables, and 
those containing only SNP and design variables. We rule these out because the design variables 
are purely our choice and individual SNP variables rarely account for more than a small fraction of the 
YIELD.  Therefore, we fit the remaining four different classes of models.

The difference between the analyses of Subsecs. \ref{Phenotype} and \ref{Design} is that the latter includes design variables. Likewise, the difference between the analyses in the Subsecs \ref{AnalSNP} and \ref{AnalPhenoSNPdsign} is that the latter includes design variables. So it is natural to 
examine the similarities and differences between the two sets of comparisons i.e., to compare the difference between Subsecs. \ref{Phenotype} and \ref{Design} to the difference between Subsecs \ref{AnalSNP} and \ref{AnalPhenoSNPdsign}. This may seem repetitive, but it is the evidence we need for the discussion in Sec. \ref{GenreConcl}.

The remainder of this section presents our new analyses of the {\sf{Wheat}} data set. We begin in Subsec. \ref{AnalWheat} with an initial descriptive data analysis. In Subsec. \ref{Phenotype} we look at two elementary analyses. First, we analyze only the  phenotypic data using our VC dimension method. We do this in two ways; for each location-year combination (i.e., environments) separately and then for the data pooled over all environments. Subsec. \ref{Design} briefly discuses the results of a re-analysis of the phenotypic data including the design variables, but the computational output is deferred to the Supplementals. In Subsec. \ref{AnalSNP}, we perform the same two analyses but combine phenotypic and SNP variables. Subsec. \ref{AnalPhenoSNPdsign} briefly discusses the results of a re-analysis using all three data types, but the computational output is deferred to the supplementals. Taken together, these are the four out of seven analyses that make sense to perform.

\subsection{Initial Data Analysis of {\sf{Wheat}} }
\label{AnalWheat}

This simple analysis of {\sf{Wheat}} will be done using only the phenotypic data, pooling over all location year combinations.
From Fig. 1 in \cite{Supplementals:2017}, we see that lines connecting locations by varieties are crossing; this indicates interactions between varieties and locations. We also see that there is a lot of variability among the varieties in Lincoln 1999 and that the YIELD varies from one location to another. The smallest value of the YIELD occurs in Lincoln in 1999 and the highest YIELD occurs at Lincoln in 2001. The most variable location is Sidney in 2000. All data points outside of the whisker plot can be considered as potential outliers but they do not appear to be excessive given the sample size. In the sequel, we focus on the Lincoln 1999 and Lincoln 2001 data because they are the most different. An alternative we did not pursue would be using the Sidney 2000 data because they are the most variable. We did this because our focus is on model selection not on data variability (although the two are related).

From Fig. 2 in \cite{Supplementals:2017}, we see that there is a reasonably linear relationship between YIELD and KPSM -- although the variance appears to increase slowly with KPSM. 
There is also a fairly good linear relationship between YIELD and SPSM although the variance increases with SPSM. However, 
we see a curvilinear relationship between YIELD and TSTWT. In addition, the variance starts small (as a function of TSTWT), increases rapidly and then appears to decrease. 
The graphs also suggest that the data do not reflect a strong relationship between 
YIELD and any of KPS, HT, and TKWT.

From Table 1 in  \cite{Supplementals:2017}, we see a strong linear relationship between YIELD and KPSM (0.93), TSTWT (0.80), and SPSM (0.74). We also see a weak linear relationship between YIELD and HT (0.04), TKWT (0.27), and KPS (0.022). 

There are some strong correlations among the phenotypic variables: These suggest collinearity. For instance, the correlation between KPSM and SPSM is 0.83, the correlation between between KPSM and TSTWT is 0.64 and the correlation between TSTWT and TKWT is 0.53. There are also some weak correlations between covariates; the correlations between TSTWT and HT, TSTWT and KPS, and HT and KPS are respectively 0.09, -0.06, and -0.22. As with the {\sf{Tour de France}} and the {\sf{Abalone}} data, whether the explanatory variables are weakly or strongly collinear does not seem to affect our methodology very much, if at all. We think this is so because the LHS of \eqref{objfn} is a {\it{difference}} of predictions. Thus even if the explanatory variables were correlated and the variance inflation were nontrivial, the expected difference only reflects location, and hence $\hat{d}_{VC}$, will be insensitive to the
elevated variability.

Intuition suggests
\begin{equation}
\label{eq5}
    YIELD = \beta_{0} + \beta_{1}\cdot TKWT\cdot KPSM + \epsilon
\end{equation}
will be a good model because YIELD is essentially the product of the number of kernels and their average weight. Likewise,
\begin{equation}
\label{eq6}
    YIELD = \beta_{0} + \beta_{1}\cdot TKWT\cdot KPS\cdot SPSM + \epsilon
\end{equation}
should also be a good model. So, to first approximation using only phenotpic variables does not lead to a {\it{unique}} good model. Indeed, these two models probably only capture the major effects of the explanatory variables. 
Both are over simplifications and we can be confident that other influences on YIELD must be considered. 
Indeed, a 3-dimensional plot of the vectors (YIELD, TKWT, KPSM) looks like a triangle that is bowed out to one side.
The bowing means that \eqref{eq5} is only an approximation; other terms are required to explain YIELD.  
Henceforth, we focus on \eqref{eq5} rather than 
\eqref{eq6} because we have limited ourselves to second order models. 

The question of which models should be on the model list has been discussed by numerous authors. \cite{Burnham:2003} provide some general guidelines assuming substantial
familiarity with the science behind a given application; this seems to be the prevailing frequentist view.
It also remains unclear how much the intuition of the key result of \cite{berk1966} 
carries over to the model selection context. 
From the Bayes perspective,
\cite{george:et:al} discusses prior selection as a way to compensate for collinearity
in explanatory variables.  \cite{Fokoue2011} find  variance-bias tradeoff on the level of mode lists for Bayes model averaging (that probably carries over to other model
averaging and selection techniques). More recently, \cite{grunwald2017} shows that in the presence of many forms of unavoidable bias, Bayes methods need to be reformulated and offers one way to do so.
Thus, although there has been ample discussion
of model list selection, there are few guidelines more specific than the injunction
to do your modeling well.  So, here we use a model list that we think will be large enough to be adequate and yet small enough that a unique best model can be found.

\subsection{Estimation of VC Dimension Using Phenotypic Covariates Only}
\label{Phenotype}

First, we analyze the data for each environment separately to show how our new 
VC dimension methodology performs in contrast to the AIC, BIC, CV $\widehat{ERM}_1$, 
$\widehat{ERM}_2$, SCAD and ALASSO. 

As will be seen, the analyses are done for two settings, location-year specific and multi-location. To implement these, we first find the order of inclusion of the phenotypic variables in the model, using correlation with YIELD. Then, we find values for $\hat{d}_{VC}$, AIC, BIC, CV,  $\widehat{ERM}_{1}$, $\widehat{ERM}_{2}$, and the models given by SCAD and
ALASSO. We present the analyses of Lincoln 1999 and Lincoln 2001 as two of the seven location-year combinations (environments), referring to \cite{Merlin:etal:2017} for the details on the other five environments. Within each environment, we pool over the design structure and the varieties. Then, for comparison purposes, we redo this form of analysis pooling over all data as was done in Subsec.\ref{AnalWheat}. We call this a multi-location analysis.

\subsubsection{Location-year analyses}
\label{locyearanal}
We begin with two examples of the use of VC dimension for fix location-year data, for Lincoln 1999 and Lincoln 2001.

\subsubsection*{Analysis of {\sf{Wheat}} data in Lincoln 1999}
\label{Lin1999}

We begin with a graphical analyses of the Lincoln 1999 data set. From Fig. 3 in \cite{Supplementals:2017}, we see that there is a strong linear relationship between YIELD and KPSM. The variance is increasing very slowly with KPSM so for practical purposes we regard it as constant. There are also some data points not close to the majority of data but the overall trend is linear. 
We see a weaker linear relationship between YIELD and SPSM and the variance appears to increase as a function of SPSM. 
We also observe that there does not appear to be any non-trivial relationship between between YIELD and any of TSTWT, KPS, HT, and TKWT.

We compare our method to the seven other methods using second order linear models in the phenotypic covariates. Thus, we use the absolute value of the correlation  between YIELD and the phenotypic covariates to order the inclusion of phenotypic variables and their products in the model. There are six linear terms, six squared terms, and fifteen cross products, this leads to a total of 27 variables and therefore 27 nested models. The order of inclusion of terms is:
TKWT $\cdot$ KPSM,  KPSM $\cdot$ HT, TSTWT $\cdot$ KPSM, SPSM $\cdot$ KPS, KPSM, KPSM$^2$, SPSM $\cdot$ KPSM, TKWT $\cdot$ SPSM, KPS $\cdot$ KPSM, SPSM, TSTWT $\cdot$ SPSM, SPSM $\cdot$ HT, SPSM$^2$, KPS $\cdot$ HT, TSTWT $\cdot$ KPS, KPS$^2$, $KPS$, TKWT $\cdot$ KPS, HT$^2$, HT, TSTWT $\cdot$ HT, TSTWT, TSTWT$^2$, TKWT$ \cdot$ HT, TKWT $\cdot$ TSTWT, TKWT$^2$, and TKWT.

\begin{table}[h]
    \caption{The column labeled size gives the number of coefficients for each of the linear models. 
    The $2^{nd}$ through $7^{th}$ columns give the corresponding estimates for $\hat{d}_{VC}$, $\widehat{ERM}_{1}$, $\widehat{ERM}_{2}$, AIC, BIC and CV for the Lincoln 1999 data. Here and elsewhere, the numbers in bold indicate the optimum for each column.}
    \label{Loc3}
    \begin{subtable}[h]{0.45\textwidth}
        \centering
        \begin{tabular}{ccccccc}
          Size& $\hat{d}_{VC}$ & $\widehat{ERM}_{1}$ & $\widehat{ERM}_{2}$ & AIC & BIC & CV\\
        \midrule
        1 & 7&    {\bf 8}&   {\bf 19}& \bf{-691}& \bf{-679}& 0.01097 \\
        2 & 7&    8&   19& -689& -673& 0.01097 \\
        3 & 7&    8&   19& -688& -668& \bf{0.01096} \\
        4 & 7&    8&   19& -687& -663& 0.01100 \\
        5 & 7&    8&   19& -693& -664& 0.01146 \\
        6 & 7&    8&   19& -691& -659& 0.01150 \\
        7 & $\bf{7}$&    8&   19& -689& -653& 0.01154 \\
        8 & 7&    8&   19& -687& -647& 0.01156  \\
        9 & 7&    8&   19& -686& -642& 0.01156 \\
        10 & 7&    8&   19& -684& -636& 0.01157 \\
        11 & 7&    8&   19& -682& -630& 0.01158 \\
        12 & 7&    8&   19& -681& -625& 0.01159 \\
        13 & 7&    8&   19& -679& -619& 0.01160 \\
        14 & 7&    8&   19& -682& -617& 0.01160 \\
        15 & 7&    8&   19& -680& -612& 0.01162 \\
        16 & 7&    8&   19& -678& -606& 0.01162  \\
        17 & 7&    8&   19& -677& -600& 0.01162\\
        18 & 7&    8&   19& -675& -594& 0.01163 \\
        \end{tabular}
        \label{Loc3set1}
    \end{subtable}
\end{table}
Our first results are given in Table \ref{Loc3}. Specifically, for Lincoln 1999, and for each of the models, we have the corresponding values of $\hat{d}_{VC}$, AIC, BIC, CV, $\widehat{ERM}_{1}$, and $\widehat{ERM}_{2}$.
It is seen that there is no variability in the estimates of $\hat{d}_{VC}$ and $\widehat{ERM}_{1}$, and $\widehat{ERM}_{2}$. The smallest difference between the size of the conjectured model and $\hat{d}_{VC}$ occurs for model 7. It is also seen that $\widehat{ERM}_{1}$, $\widehat{ERM}_2$, AIC and BIC pick the smallest model i.e., the one with $TKWT\cdot KPSM$ as the only variable. (When using $\widehat{ERM}_{2}$ or $\widehat{ERM}_{2}$, the convention is to choose the models giving their smallest values.) Although, the CV values vary only slightly, they
exhibit a strong pattern hinting that the model 
consisting of the first three terms should be chosen.

Now, we turn to shrinkage methods. The optimal models chosen by SCAD and ALASSO are, respectively,

\begin{equation}
\label{Loc3SCAD}
\widehat{YIELD} =  1.93 + 0.52\cdot KPSM\cdot TKWT,
\end{equation}
and 
\begin{equation}
\label{Loc3ALASSO}
\widehat{YIELD} =  1.93 + 0.48\cdot KPSM\cdot TKWT + 0.038\cdot SPSM\cdot KPS + 0.0043\cdot TKWT.
\end{equation}
Thus, the model chosen by ALASSO uses the first, fourth and $27^{th}$ explanatory variables. However, the extra terms beyond TKWT$\cdot$KPSM have very small coefficients. (All variables have been studentized so they are on the same scale.)
So, we regard \eqref{Loc3ALASSO} as effectively the same as \eqref{Loc3SCAD}.
So, we are left with the one term model (more or less discounted in Subsec. \ref{AnalWheat}), the three term CV model, and the seven term $\hat{d}_{VC}$ model. It is hard to decide between the three and seven term model because both are a priori plausible. However, note that CV is a predictive criterion while VC is a complexity criterion.  Thus if the goal is merely predictive, the CV's model is preferred and if the goal is model identification, the $\hat{d}_{VC}$ model is preferred. This follows because the best finite sample predictor may be much simpler than the predictor from the true model when it is complex.

\subsubsection*{Analysis of {\sf{Wheat}} Data in Lincoln 2001}
\label{AnalLin2001}

Starting with a graphical analysis of the Lincoln 2001 data, we see from Fig. 4 in \cite{Supplementals:2017} 
that there is a reasonably linear relationship between YIELD and KPSM. The variability does not change much as KPSM increases. There are some data points that are not close to  the majority of the data, but the overall trend is linear. 
We also see a weak linear relationship between YIELD and SPSM and the variance appears to increase as a function of SPSM.
There does not appear to be any non-trivial relationship between YIELD and the other phenotypic covariates.

For the Lincoln 2001 data set, we use the same model list i.e., all terms of order $\leq 2$ in the phenotypic variables. Under the absolute value of the correlation with $YIELD$, the order of inclusion is:
 TKWT $\cdot$ KPSM, TSTWT $\cdot$ KPSM, KPSM, SPSM $\cdot$ KPS, KPSM$^2 $,
 SPSM $\cdot$ KPSM, KPSM $\cdot$ HT, TKWT $\cdot$ SPSM, KPS $\cdot$ KPSM, TSTWT $\cdot$ SPSM,
 SPSM, SPSM$^2$, SPSM $\cdot$ HT, TSTWT, TSTWT$^2$, TKWT $\cdot$ KPS,
 TSTWT $\cdot$ KPS, KPS$^2$, TKWT $\cdot$ TSTWT, KPS, TKWT, TKWT$^2 $, KPS $\cdot$ HT , 
 TKWT $\cdot$ HT, TSTWT $\cdot$ HT,  HT$^2$,  HT. This ordering on the phenotypic variables is different from that of Lincoln 1999.

\begin{table}[ht]
    \caption{The column labeled size gives the number of coefficients for each of the linear model. 
    The $2^{nd}$ through $7^{th}$ columns give the corresponding estimates for $\hat{d}_{VC}$, $\widehat{ERM}_{1}$, $\widehat{ERM}_{2}$, AIC, BIC and VC for the Lincoln 2001.}
    \begin{subtable}[h]{0.45\textwidth}
        \centering
        \begin{tabular}{ccccccc}
          Size& $\hat{d}_{VC}$ & $\widehat{ERM}_{1}$ & $\widehat{ERM}_{2}$ & AIC & BIC & CV \\
        \hline 
        1&  6&    \bf{5}&   \bf{13}& \bf{-1055}& \bf{-1043}& \bf{0.003523772} \\
		2&  6&    5&   13& -1053& -1037& 0.003523852 \\
		3&  6&    5&   13& -1051& -1032& 0.003528050 \\
		4&  6&    5&   13& -1050& -1026& 0.003531868 \\
		5&  6&    5&   13& -1048& -1020& 0.045582343 \\
		6&  \bf{6}&    5&   13& -1046& -1015& 0.042056585 \\
		7&  6&    5&   13& -1044& -1009& 0.041950879 \\
		8&  6&    5&   13& -1042& -1003& 0.042716568 \\
		9&  6&    5&   13& -1040&  -997& 0.042944324
        \end{tabular}
        \label{Loc2set1}
    \end{subtable}
    \label{Loc2}
\end{table}

Parallel to Table \ref{Loc3}, the columns of Table \ref{Loc2} give the model sizes and the their corresponding values of $\hat{d}_{VC}$, $\widehat{ERM}_{1}$, $\widehat{ERM}_{2}$, AIC, BIC, and CV but for the Lincoln 2001 data. First, we see that $\hat{d}_{VC} = 6$ no matter the size of the conjectured model. (In fact, $\hat{d}_{VC}$ for model 26 was 8 but we regarded that as a random fluctuation.) In any case, the smallest difference between the size of the conjectured model and $\hat{d}_{VC}$ occurs for the conjectured model of size $6$.
Likewise, we also observe essentially no variability in the values of $\widehat{ERM}_{1}$ and $\widehat{ERM}_{2}$, so following convention, they indicate the one term model. 
The smallest values for  AIC, BIC, and CV also occur for the model with one explanatory variable. The main
 reason we take the CV seriously is that its values are monotonically increasing 
even though the differences are very small.

Turning to sparsity methods, the model chosen by SCAD is
\begin{equation}
\label{Loc2SCAD}
\widehat{YIELD} = 4.67 + 0.40\cdot TKWT \cdot KPSM,
\end{equation}
and the model chosen by ALASSO is 
\begin{equation}
\label{Loc2ALASSO}
\widehat{YIELD} = 4.67 + 0.40\cdot TKWT \cdot KPSM + 0.0017\cdot KPSM.
\end{equation}
It is seen that the models in \eqref{Loc2SCAD} and \eqref{Loc2ALASSO} differ only by one term with a very small coefficient and so are the same for all practical purposes.

Taken together, we see that  $\hat{d}_{VC}$ chooses the 6 term model and all the other methods choose the one term model.
As described at the end of Subsec. \ref{AnalWheat}, we find this model too much of an oversimplification to be credible. 
Thus in this case, the only reasonable answer is given by the VC dimension model.

\subsubsection{Multi-location Analysis}
\label{MultiLoc}

For this analysis, we pooled over location-year combinations, design structures and varieties. An initial analysis of the data was given in Subsec. \ref{AnalWheat}. So, here, we merely proceed with the model selection problem. 
 
 Under absolute value of correlation with YIELD, the order of inclusion of the explanatory variables is:
 TKWT $\cdot$ KPSM, TSTWT $\cdot$ KPSM, KPSM, SPSM $\cdot$ KPS, KPSM$^2$, KPSM $\cdot$ HT, 
 TKWT $\cdot$ SPSM, TSTWT$^2$, TSTWT, SPSM $\cdot$ KPSM,  KPS $\cdot$ KPSM, TSTWT $\cdot$ SPSM,
 SPSM SPSM $\cdot$ HT, SPSM$^2$, TKWT $\cdot$ TSTWT, TKWT, TSTWT $\cdot$ HT,
 TKWT$^2$ TKWT $\cdot$ KPS, TSTWT $\cdot$ KPS, TKWT $\cdot$ HT, KPS $\cdot$ HT, HT, HT$^2$ KPS, KPS$^2$.
 As in Secs. \ref{Lin1999} and \ref{AnalLin2001},  we consider 27 nested models. This leads to Table \ref{MultiLocAnal} and is parallel to 
 Tables \ref{Loc3} and \ref{Loc2}.

\begin{table}[ht]
    \caption{The column labeled size gives the number of coefficients for each of the linear model. 
    The $2^{nd}$ through $7^{th}$ columns give the corresponding estimates for $\hat{d}_{VC}$, $\widehat{ERM}_{1}$, $\widehat{ERM}_{2}$, AIC, BIC, and CV for $\sf{Wheat}$.}
    \begin{subtable}[h]{0.45\textwidth}
        \centering
        \begin{tabular}{ccccccc}
          Size& $\hat{d}_{VC}$ & $\widehat{ERM}_{1}$ & $\widehat{ERM}_{2}$ & AIC & BIC & CV\\
        \midrule
        1&  13&    \bf{7}&   \bf{11}& \bf{-9160}& \bf{-9142}& \bf{0.001801809} \\
		2&  14&    7&   11& -9159& -9136& 0.001802107 \\
		3&  13&    7&   11& -9159& -9130& 0.001802470 \\
		4&  13&    7&   11& -9159& -9124& 0.001803912 \\
		5&  13&    7&   11& -9159& -9118& 0.001803933 \\
		6&  14&    7&   11& -9157& -9110& 0.001805230 \\
		7&  14&    7&   11& -9156& -9103& 0.001805349 \\
		8&  14&    7&   11& -9158& -9099& 0.001806966 \\
		9&  14&    7&   11& -9156& -9091& 0.001807176\\
        10& 14&    7&   11& -9154& -9084& 0.001808455 \\
		11& 13&    7&   11& -9153& -9076& 0.001808248 \\
		12& 13&    7&   11& -9151& -9069& 0.001808966 \\
		13& 14&    7&   11& -9150& -9061& 0.001810670 \\
		14& \bf{14}&    7&   11& -9148& -9054& 0.001812884 \\
		15& 13&    7&   11& -9147& -9047& 0.001813791 \\
		16& 13&    7&   11& -9145& -9039& 0.001818615 \\
		17& 13&    7&   11& -9144& -9032& 0.001819879 \\
		18& 13&    7&   11& -9144& -9027& 0.001823882
        \end{tabular}
        \label{Loc8set1}
    \end{subtable}
    \label{MultiLocAnal}
\end{table}

First, we note there is no variability in $\widehat{ERM}_{1}$ and $\widehat{ERM}_{2}$ so following standard usage, they select a model with TKWT $\cdot$ KPSM as the only explanatory variable. Likewise, AIC, BIC, and CV suggest the one term model. However, $\hat{d}_{VC} = 14$ means the VC dimension chooses the model with the first 14 terms from the ordered list. SCAD and ALASSO both give the one term model
\begin{equation}
\label{Loc8SCAD}
    \widehat{YIELD}  = 3.43 + 1.12\cdot TKWT\cdot KPSM.
\end{equation}
the same as $\widehat{ERM}_{1}$, $\widehat{ERM}_{2}$, AIC, BIC and CV. 
Thus the only reasonable model is the one chosen by $\hat{d}_{VC}$ and we note that the model here has 14 terms whereas when the data were simpler in Subsubsec \ref{locyearanal} the $\hat{d}_{VC}$ model had only six terms in accord with intuition.

\subsection{Analysis of {\sf{Wheat}} Using Phenotypic Data and the Design Structure}
\label{Design}

Our objective in this section is to take the design structure into account and see its impact on the values of the VC dimension and hence on the chosen model. As before, we implement our method on the Lincoln 1999 and Lincoln 2001 data sets.
Including design variables forces us to use a more complicated bootstrap procedure that would otherwise be sufficient. Thus, to implement our method here, we perform a {\it{restricted}} bootstrap. Specifically, we bootstrap in each level of the design variable (incomplete block) so that each half data set has all levels of the design structure. We do this to maintain the design structure and its effects.
Since the blocking variable has 32 different unordered categories, we cannot use it to compute the correlation. So to include phenotypic variables in the models, we use the order of inclusion from the analyses of the Lincoln 1999, 2001 and multi-location data sets. For each model, we then include the block effect. For instance, in Lincoln 1999, the first term that gets into the model is TKWT$\cdot$ KPSM and the model that we fit has TKWT$\cdot$KPSM plus $IBLK_{j}, j = 1,2, \cdots, 32$ ($IBLK_{j}$ is the $j^{th}$ incomplete block). So, the size of each candidate model increases by 32.

\subsubsection{Estimation of the VC dimension using phenotypic covariates and the design structure }
 \label{envdesign}
In this section, we estimate the VC dimension by combining the phenotypic covariates and  the design structure (incomplete block) for Lincoln 1999 and Lincoln 2001.

\subsubsection*{Analysis of Lincoln 1999 data with the design structure}
\label{Anal199design}
Since the variable (incomplete block) representing the design is a class variable, we did not find its correlation with YIELD. So, we use the same order of inclusion as in the analysis of Lincoln 1999 without taking the design structure into account.

\begin{table}[h]
    \caption{Estimation of $\hat{d}_{VC}$, $\widehat{ERM}_{1}$, $\widehat{ERM}_{1}$, AIC and BIC in Lincoln 1999 using the design structure.}
    \begin{subtable}[h]{0.45\textwidth}
        \centering
        \begin{tabular}{ccccccc}
          Size& $\hat{d}_{VC}$ & $\widehat{ERM}_{1}$ & $\widehat{ERM}_{2}$ & AIC & BIC & CV \\
        \midrule
        1&  7&    \bf{8}&   \bf{19}& \bf{-685}& \bf{-661}& 0.01097206 \\
		2&  7&    8&   19& -684& -655& 0.01096831 \\
		3&  7&    8&   19& -682& -650& \bf{0.01096340} \\
		4&  7&    8&   19& -681& -645& 0.01100318 \\
		5&  7&    8&   19& -687& -646& 0.01146253 \\
		6&  7&    8&   19& -685& -641& 0.01150139 \\
		7&  \bf{7}&    8&   19& -683& -635& 0.01153652 \\
		8&  7&    8&   19& -681& -629& 0.01155855 \\
		9&  7&    8&   19& -680& -624& 0.01155742
        \end{tabular}
        \label{Loc1set11}
    \end{subtable}
    \label{Loc11}
\end{table}
It is easy to see that Tables \ref{Loc11} and \ref{Loc3} are qualitatively identical. That is the design structure had essentially no effect on $\hat{d}_{VC}$, $\widehat{ERM}_{1}$, $\widehat{ERM}_{1}$, AIC, BIC, or CV.
Likewise, the optimal model chosen by SCAD is the same as \eqref{Loc3SCAD}. The model chosen by ALASSO is 
\begin{equation}
\label{DesignALASSO01}
    \widehat{YIELD}  = 1.93 + 0.48\cdot KPSM \cdot TKWT + 0.04\cdot SPSM \cdot KPS,
\end{equation}
trivially different from \eqref{Loc3ALASSO}. Our conclusions are therefore the same as in Subsec. \ref{Lin1999}

\subsubsection*{Analysis of Lincoln 2001 data set using phenotypic covariates and the design structure}
\label{desinlin2001}

The inclusion of covariates is the same here as in Subsubsec. \ref{AnalLin2001}.
Parallel to Table \ref{Loc11}, the columns of Table \ref{Loc12} give the candidate model sizes and their corresponding $\hat{d}_{VC}$, $\widehat{ERM}_{1}$, $\widehat{ERM}_{2}$, AIC, BIC, and CV values.

\begin{table}[h]
    \caption{Estimation of $\hat{d}_{VC}$, $\widehat{ERM}_{1}$, $\widehat{ERM}_{1}$, AIC and BIC in Lincoln 2001 using the design structure.}
    \begin{subtable}[h]{0.45\textwidth}
        \centering
        \begin{tabular}{ccccccc}
          Size& $\hat{d}_{VC}$ & $\widehat{ERM}_{1}$ & $\widehat{ERM}_{2}$ & AIC & BIC & CV\\
        \midrule
        1&  7&    \bf{5}&   \bf{14}& \bf{-1025}& \bf{-891}& \bf{0.003523772} \\
		2&  7&    5&   14& -1023& -886& 0.003523852 \\
		3&  7&    5&   14& -1022& -880& 0.003528050 \\
		4&  7&    5&   14& -1021& -876& 0.003531868 \\
		5&  7&    5&   14& -1019& -870& 0.045582343 \\
		6&  7&    5&   14& -1019& -866& 0.042056585 \\
		7&  \bf{7}&    5&   14& -1017& -860& 0.041950879 \\
		8&  7&    5&   14& -1017& -856& 0.042716568 \\
		9&  7&    5&   14& -1016& -851& 0.042944324
        \end{tabular}
        \label{Loc1set112}
    \end{subtable}
    \label{Loc12}
\end{table}
It is easy to see that Tables \ref{Loc2} and \ref{Loc12} are qualitatively identical except that here $\hat{d}_{VC} =7$ whereas $\hat{d}_{VC} = 6$ in Table \ref{Loc12} a difference that can be ascribed to random variation. Likewise, the model chosen by SCAD is the same as \eqref{Loc2SCAD}.
The model chosen by ALASSO is
\begin{equation}
\label{DesignALASSO2001}
    \widehat{YIELD}  = 4.67 + 0.4\cdot KPSM \cdot TKWT,
\end{equation}
trivially different from \eqref{Loc2ALASSO}.
Our conclusions are therefore the same as in Subsubsec. \ref{AnalLin2001}: The design structure had no effect on the model selection techniques and $\hat{d}_{VC}$, gave the only plausible model. 

\subsubsection{Multi-location Analysis Using the Phenotypic Covariates and Design Structure}
\label{sect42}

For this analysis, we pooled over location-year combinations and varieties. 
In the model, we included the incomplete block reflecting the design variables.
This is the same analysis as was performed in Subsubsec. \ref{MultiLoc}, but now we are including the design structure in the models on the model list. Here, the order of inclusion of variables is the same as in Subsubsec. \ref{MultiLoc}.

\begin{table}[h]
    \caption{ The column labeled size gives the number of coefficients for each model of the linear model. The $2^{nd}$ through $7^{th}$ columns give the corresponding estimates for $\hat{d}_{VC}$, $\widehat{ERM}_{1}$, $\widehat{ERM}_{1}$, AIC, BIC, and CV for multi-location analysis with design structure.}
    \begin{subtable}[h]{0.45\textwidth}
        \centering
        \begin{tabular}{ccccccc}
          Size& $\hat{d}_{VC}$ & $\widehat{ERM}_{1}$ & $\widehat{ERM}_{2}$ & AIC & BIC & CV \\
        \midrule
        1&  13&    \bf{7}&   \bf{11}& \bf{-9160}& \bf{-9142}& \bf{0.001801809} \\
		2&  13&    7&   11& -9159& -9136& 0.001802107 \\
		3&  13&    7&   11& -9159& -9130& 0.001802470 \\
		4&  13&    7&   11& -9159& -9124& 0.001803912 \\
		5&  13&    7&   11& -9159& -9118& 0.001803933 \\
		6&  13&    7&   11& -9157& -9110& 0.001805230 \\
		7&  13&    7&   11& -9156& -9103& 0.001805349 \\
		8&  13&    7&   11& -9158& -9099& 0.001806966 \\
		9&  13&    7&   11& -9156& -9091& 0.001807176 \\
		10& 13&    7&   11& -9154& -9084& 0.001808455 \\
		11& 13&    7&   11& -9153& -9076& 0.001808248 \\
		12& 13&    7&   11& -9151& -9069& 0.001808966 \\
		13& \bf{13}&    7&   11& -9150& -9061& 0.001810670 \\ 
		14& 13&    7&   11& -9148& -9054& 0.001812884 \\
		15& 13&    7&   11& -9147& -9047& 0.001813791 \\
		16& 13&    7&   11& -9145& -9039& 0.001818615 \\
		17& 13&    7&   11& -9144& -9032& 0.001819879 \\
		18& 13&    7&   11& -9144& -9027& 0.001823882
        \end{tabular}
        \label{MultiLoc1set112}
    \end{subtable}
    \label{MultiLoc12}
\end{table}
The natural comparison is between Tables \ref{MultiLoc12} and \ref{MultiLocAnal}. It is seen that apart from random variation, they are identical.
Moreover, the sparsity methods give exactly the same results in both settings.  Thus the conclusions here are the same as in Subsubsec. \ref{MultiLoc}, namely the design variables have no impact on model selection and $\hat{d}_{VC}$ gave the only plausible model.

\subsection{Analysis of {\sf{Wheat}}  using the PhenotypicData and SNP Covariates}
\label{AnalSNP}

Our goal in this section is to see how using both the phenotypic and the SNP data will affect the estimate of VC dimension and therefore the model chosen. Since the variables representing the SNP's have missing data, we simply dropped all rows with missing SNP values and only used SNP's that were complete, i.e., no missing values, meaning we dropped columns as well. While this is not reasonable in general for data analysis, our point here is merely to demonstrate the use of VC dimension for model selection. Thus we retained only 6 SNP's (barc67, cmwg680bcd366, bcd141, barc86, gwm155, barc12). In Subsec. \ref{SNPModel} we do model selection using the phenotypic and SNP variables for Lincoln 1999 and Lincoln 2001.  In Subsec. \ref{AnalSNPMultiLoc}, we perform the corresponding multi-location analysis.

\subsubsection{Location-year Analyses Using Phenotypic and SNP Covariates}
\label{SNPModel}

In this section,  we estimate the VC dimension using phenotypic and SNP variables for the Lincoln 1999 data and then for the Lincoln 2001 data. Analogous analyses can be done for other location-year combinations but are omitted here.

\subsubsection*{Analysis of {\sf{Wheat}} in Lincoln 1999}
\label{AnalSNPLin99}
As we have done before, to estimate the VC dimension, we order the inclusion of variables in the model. We use all first and second order terms in the six phenotypic variables and all first terms in the 6 SNP. Together, this gives 33 terms. Using the absolute value of the correlation with YIELD, the order of inclusion of variables is: TKWT $\cdot$ KPSM, SPSM$^2$, KPSM, TSTWT $\cdot$ KPS, KPSM$^2$, KPS $\cdot$ HT, SPSM $\cdot$ KPS, $KPS^2$, TKWT $\cdot$ SPSM, SPSM, SPSM $\cdot$ KPSM, TSTWT$^2$, TSTWT $\cdot$ HT, KPS $\cdot$ KPSM, TSTWT $\cdot$ SPSM, KPS, SPSM $\cdot$ HT, TKWT$\cdot$ KPS, HT, KPSM $\cdot$ HT, TSTWT $\cdot$ KPSM, gwm155, HT$^2$ TSTWT, TKWT $\cdot$ HT,  bcd141, TKWT$\cdot$ SPSM, barc86, cmwg680bcd366, barc67, TKWT $\cdot$ TSTWT, TKWT$^2$, TKWT. Here the SNP's are simply denoted by their labels in the data.
The estimates of $\hat{d}_{VC}$, $\widehat{ERM}_{1}$, $\widehat{ERM}_{2}$, AIC, BIC, and CV are given in Table \ref{SNPLin1999}.

\begin{table}[ht]
    \caption{The column labeled size gives the number of coefficients for each of the linear model. 
    The $2^{nd}$ through $7^{th}$ columns give the corresponding estimates for $\hat{d}_{VC}$, $\widehat{ERM}_{1}$, $\widehat{ERM}_{2}$, AIC, BIC, and CV for $\sf{Wheat}$ in Lincoln 1999 using phenotypic and SNP variables.}
    \begin{subtable}[h]{0.45\textwidth}
        \centering
        \begin{tabular}{ccccccc}
          Size& $\hat{d}_{VC}$ & $\widehat{ERM}_{1}$ & $\widehat{ERM}_{2}$ & AIC & BIC & CV \\
        \midrule
		1&  13&    \bf{9}&   \bf{28}& -545& \bf{-533}& \bf{0.01390512}\\
		2&  13&    9&   28& -543& -527& 0.01390611\\
		3&  13&    9&   28& \bf{-548}& -528& 0.01452534\\
		4&  13&    9&   28& -546& -523& 0.01459774\\
		5&  13&    9&   28& -546& -518& 0.01459255\\
		6&  13&    9&   28& -544& -512& 0.01461276\\
		7&  13&    9&   28& -542& -506& 0.01468128\\
		8&  13&    9&   28& -541& -501& 0.01468513\\
		9&  13&    9&   28& -539& -496& 0.01470951\\
		10& 13&    9&   28& -537& -490& 0.01471344\\
        11& 13&    9&   28& -537& -486& 0.01473980\\
		12& 13&    9&   28& -535& -480& 0.01480711\\
		13& \bf{13}&    9&   28& -533& -474& 0.01483929\\
		14& 13&    9&   28& -535& -472& 0.01483715\\
		15& 13&    9&   28& -533& -467& 0.01487698\\
		16& 13&    9&   28& -532& -462& 0.01491353\\
		17& 13&    9&   28& -530& -456& 0.01492661\\
		18& 13&    9&   28& -529& -450& 0.01494602\\
		19& 13&    9&   28& -528& -446& 0.01512149\\
		20& 13&    9&   28& -527& -441& 0.01506091
        \end{tabular}
        \label{Lin991}
    \end{subtable}
    \label{SNPLin1999}
\end{table}
Table \ref{SNPLin1999} shows that $\hat{d}_{VC} = 13$, AIC chooses the three term model and $\widehat{ERM}_{1}$, $\widehat{ERM}_{2}$, BIC, and CV all choose the model with one explanatory variable TKWT $\cdot$ SPSM. 
SCAD gives the model
\begin{equation}
\label{SNPLin99SCAD}
    \widehat{YIELD}  = 1.90 + 0.49\cdot TKWT\cdot KPSM
\end{equation}
and ALASSO gives the model
\begin{equation}
\label{Loc8ALASSO}
    \widehat{YIELD}  = 1.9 + 0.46\cdot TKWT\cdot KPSM + 0.02\cdot SPSM \cdot KPS.
\end{equation}

Since we have ruled out the one term model, we are left with the 13 term model from $\hat{d}_{VC}$, the three term model from AIC, and the the two term model from ALASSO. The extra term in the ALASSO model has a very small coefficient, and so is essentially indistinguishable from the one term model. So eliminating it from further consideration leaves only the VC dimension and AIC models as possibly reasonable. Note that neither include any SNP's and it is hard to argue that either model is better. For instance, the extra 10 terms in the VC dimension model may or may not contribute more variation than any increase in bias from using the 3 term AIC model. On the other hand, we see that including SNP's in the model list affects the size of the models chosen cf. Table \ref{Loc3} in Subsec \ref{Lin1999}. On an intuitive basis, we can recall that AIC is associated with prediction in contrast to 
$d_{VC}$ which is a complexity criterion. Thus the comments at the end of Subsubsec. \ref{Lin1999} apply here.

\subsubsection*{Analysis of {\sf{Wheat}} in Lincoln 2001}
\label{AnalSNPLin01}

The analysis here is similar to what we have done before. To estimate the VC dimension, we first nest our model list using correlation. The order of inclusion of variables is: 
TKWT $\cdot$ KPSM, TSTWT $\cdot$ KPSM, KPSM, SPSM $\cdot$ KPS, KPSM$^2$,     
SPSM $\cdot$ KPSM, TKWT $\cdot$ SPSM, TSTWT $\cdot$ SPSM, SPSM, KPSM $\cdot$ HT, SPSM$^2$, KPS $\cdot$ KPSM, SPSM $\cdot$ HT, TSTWT, TSTWT$^2$, barc67, TKWT $\cdot$ TSTWT, barc86, TKWT, TKWT$^2$, cmwg680bcd366, bcd141, TKWT $\cdot$ KPS, gwm155, TSTWT $\cdot$ KPS, barc12, KPS$^2$, KPS, TKWT $\cdot$ HT, KPS $\cdot$ HT, TSTWT $\cdot$ HT, HT, HT$^2$. So, we consider 33 nested models.
Table 2 in \cite{Supplementals:2017} shows that any value of $\hat{d}_{VC}$ from 31 to 33 may be reasonable. Table 2 in \cite{Supplementals:2017} shows that $\widehat{ERM}_{1}$, AIC, BIC, and CV suggest the one term model, while $\widehat{ERM}_2$ suggests the 5 term model.
SCAD chooses the model 
\begin{eqnarray}
\widehat{YIELD} = 4.66 + 0.39 \cdot TKWT\cdot KPSM,
\end{eqnarray}
but we were unable to identify a model using ALASSO because of convergence problems  in the implementation of the code we were using.

In this more complicated setting i.e., including SNP's,  we therefore have one method, VC dimension, that chooses the largest model, four methods that choose the smallest model, one method that chooses the 5 term model, and one method, ALASSO, that fails to choose a model at all.
As before, we rule out the smallest model because we know some of the SNP's have an effect on YIELD. We also rule out the 5 term model because it has no SNP's.
So, by default, the VC dimension chosen models are the most credible. 

Overall, however, one can argue that the data are
so complex that all the methods are breaking down:  Choosing the smallest model or the largest model
are instances of the bail-out or bail-in effect, respectively, noted in  \cite{clarke2013}. In such cases, the
largest models are the ones to be chosen but with caution because the bail-in effect suggests that even the largest models are missing important terms.   
Again, the VC dimension technique is giving a  better result than
other the other techniques but none of them are really satisfactory.  In model selection, one tries to 
choose a model list so the true model, assuming it exists, does not lie on the boundary
of the list but is a sort of `interior point'. Thus this analysis suggests that the model list we have used is inadequate.

\subsubsection{Multi-location Analysis Using the Phenotypic and SNP Covariates}
\label{AnalSNPMultiLoc}

In this subsection, we pooled over all locations and dropped all variables or observations with missing data points. The sample size was thereby reduced to 2631. As before, we ordered the inclusion of covariates in our model by the absolute value of the correlation. The order of inclusion of terms is as follows:
TKWT $\cdot$ KPSM,  TSTWT $\cdot$ KPS,  KPSM, SPSM$^2$, KPS $\cdot$ HT,    
KPSM $^2$, TKWT $\cdot$ SPSM, TKWT $\cdot$ HT , TSTWT , SPSM $\cdot$ KPS , 
  KPS$^2$, TSTWT$^2$,  SPSM,  SPSM $\cdot$ KPSM, TSTWT $\cdot$ HT,  
  TKWT $\cdot$ TSTWT, TKWT ,  TSTWT$\cdot$ KPSM, TKWT$^2$, TKWT $\cdot$ KPS , 
  TSTWT $\cdot$ SPSM, TKWT $\cdot$ SPSM, HT$^2$ , KPS $\cdot$ KPSM,   HT,           
  barc67, KPSM$ \cdot$ HT, cmwg680bcd366, bcd141, barc86,     
  gwm155, KPS, SPSM $\cdot$ HT. 
With this order, we fit 32 different models because the most complex model with 33 terms did not converge, possibly because of muli-colinearity. 

Table 3 in \cite{Supplementals:2017} shows that, roughly, $\hat{d}_{VC}$ is suggesting the model of size 24 may be 
most reasonable but there is little apttern to the values sugesting any inference is unreliable. Table 3 in \cite{Supplementals:2017} also suggests AIC, BIC and CV are choosing a one term model.
The other two complexity methods,  $\widehat{ERM}_{1}$
and $\widehat{ERM}_{2}$, formally suggest model sizes of 7 and 15, resp.
However, the values in the column look more like fandom fluctuations rather than a real signal.
Consequently, these methods may be suggesting the one-term model.  The fact that they are not giving
a clear answer means they are breaking down.
SCAD and ALASSO choose the one term model and have identical parameter estimates. That is, both lead to 
\begin{equation}
\label{DesignSCADSNIP}
    \widehat{YIELD}  = 3.43 + 1.12\cdot TKWT\cdot KPSM.
\end{equation}

Thus, $\widehat{ERM}_{1}$, AIC, BIC, CV, SCAD, and ALASSO pick models we don't believe while $\widehat{ERM}_{1}$, $\widehat{ERM}_{2}$, and $d_{VC}$ are not giving a clear answers.  As in 
Subsec. \ref{AnalSNPMultiLoc}, it seems likely that the model list is inadequate.   This is a
common result when dealing with real data.

\subsection{Analysis of {\sf{Wheat}} Using Phenotypic and SNP Covariates, and the Design Structure}
\label{AnalPhenoSNPdsign}

Our objective in this section is to take the phenotypic variables, the SNP variables and the design variables
into account and assess the impact of the design variables on the VC dimension, as a contrast to the analyses of Subsec. \ref{AnalSNP}. As before, we implement our method on the Lincoln 1999 and Lincoln 2001 data sets and, as in Subsubsec. \ref{sect42}, we used a restricted bootstrap for the multi-location data.  The blocking variable continues to have 32 different unordered categories, so we cannot use it to compute correlations. However, the SNP variables are numerical so we use them as well when we order the inclusion of covariates in the model. So, to include phenotypic variables and SNP variables in the models, we use the order of inclusion from the analyses of the Lincoln 1999, the Lincoln 2001 and the multi-location data sets of Subsec. \ref{AnalSNP}. Also as in Subsec. \ref{Design}, for each model, we include the block effect so the size of each candidate model increases by 32.
In Subsubsec. \ref{AnlaphenoSNPdesignloc}, we do model selection using the phenotypic, SNP and design variables for Lincoln 1999 and Lincoln 2001. In Subsubsec. \ref{AnlaphenoSNPdesignmultiloc}, we perform a multi-location analysis.

\subsubsection{Analysis of {\sf{Wheat}} Using Phenotypic and SNP Covariates, and the Design Structure in Each Environment}
\label{AnlaphenoSNPdesignloc}
Again, we estimate the VC dimension using  phenotype, SNP and design covariates combined for Lincoln 1999 and Lincoln 2001. The same analysis can be done for the other environments but they are omitted here.

\subsubsection*{Analysis of {\sf{Wheat}} in Lincoln 1999}
\label{analphenosnpdesign1999}

The order of inclusion of terms here is the same as in Sec. \ref{AnalSNPLin99}.

\begin{table}[h]
    \caption{The column labeled size gives the number of coefficients for each of the linear model. 
    The $2^{nd}$ through $7^{th}$ columns give the corresponding estimates for $\hat{d}_{VC}$, $\widehat{ERM}_{1}$, $\widehat{ERM}_{2}$, AIC, BIC, and CV for $\sf{Wheat}$ in Lincoln 1999 using phenotypic, SNP and design variables.}
    \begin{subtable}[h]{0.45\textwidth}
        \centering
        \begin{tabular}{ccccccc}
          Size& $\hat{d}_{VC}$ & $\widehat{ERM}_{1}$ & $\widehat{ERM}_{2}$ & AIC & BIC & CV \\
        \midrule
        1&  13&    \bf{9}&   \bf{28}& -539& \bf{-515}& \bf{0.01390512}\\
		2&  13&    9&   28& -537& -509& 0.01390611\\
		3&  13&    9&   28& \bf{-542}& -510& 0.01452534\\
		4&  13&    9&   28& -541& -505& 0.01459774\\
		5&  13&    9&   28& -540& -500& 0.01459255\\
		6&  13&    9&   28& -538& -495& 0.01461276\\
		7&  13&    9&   28& -536& -489& 0.01468128\\
		8&  13&    9&   28& -535& -484& 0.01468513\\
		9&  13&    9&   28& -533& -478& 0.01470951\\
		10& 13&    9&   28& -531& -472& 0.01471344\\ 
        11& 13&    9&   28& -531& -468& 0.01473980\\
		12& 13&    9&   28& -529& -462& 0.01480711\\
		13& \bf{13}&    9&   28& -527& -456& 0.01483929\\
		14& 13&    9&   28& -529& -455& 0.01483715\\
		15& 13&    9&   28& -527& -449& 0.01487698\\
		16& 13&    9&   28& -526& -444& 0.01491353\\
		17& 13&    9&   28& -524& -438& 0.01492661\\
		18& 13&    9&   28& -523& -432& 0.01494602\\
		19& 13&    9&   28& -522& -428& 0.01512149\\
		20& 13&    9&   28& -521& -423& 0.01506091\\
        \end{tabular}
        \label{phenosnpdesiGene1}
    \end{subtable}
    \label{phenosnpdesiGene}
\end{table}

Examination of Tables \ref{phenosnpdesiGene} and \ref{SNPLin1999} shows that they are qualitatively identical. That is, the design structure has essentially no effect on $\hat{d}_{VC}$, $\widehat{ERM}_{1}$, $\widehat{ERM}_{2}$, AIC, BIC, and CV, for $\sf{Wheat}$. Likewise, the model chosen by SCAD is the same as \eqref{SNPLin99SCAD}. The model chosen by ALASSO is 
\begin{equation}
    \label{phenosnpdesignalasso}
    \widehat{YIELD} = 1.90 + 0.46\cdot TKWT\cdot KPSM + 0.02\cdot SPSM \cdot KPS,
\end{equation}
trivially different from \eqref{Loc8ALASSO}. Our conclusions are therefore the same as in Subsec. \ref{AnalSNPLin99}.

\subsubsection*{Analysis of {\sf{Wheat}} in Lincoln 2001}
\label{phenosnpdesignanal2001}

The analysis here is similar to that of Subsubsec. \ref{AnalSNPLin01}. 
For this class of models, we fit 32 models because the most complex model with 33 variables did not converge. The models suggested by $\hat{d}_{VC}$, $\widehat{ERM}_{1}$, $\widehat{ERM}_{2}$, AIC, BIC, and CV are the model with 33, 1, 7, 1, 1, and 1 terms respectively in Table 4 in \cite{Supplementals:2017}. 
It is natural to compare these results with those from Table 2 from \cite{Supplementals:2017}. Table 2 suggests model with 33, 1, 5, 1, 1, 1 terms respectively. So, neglecting random variation, these results are qualitatively identical.
Thus, the inclusion of design variables has no effect on the models chosen. The SCAD model is also the same as in Subsubsec. \ref{AnalSNPLin01}, but our implementation of ALASSO did not converge. Nevertheless, overall, the conclusions we found in Subsubsec. \ref{AnalSNPLin01} remain true here. 

\subsubsection{Multi-location Analysis with Phenotypic and SNP Covariates and the Design Structure}
\label{AnlaphenoSNPdesignmultiloc}

As in Subsubsec. \ref{AnalSNPMultiLoc}, we pooled over locations and dropped all variables or observations with missing data values thereby reducing the sample size to 2631. Also, the order of inclusion of variables in the model is the same as in Subsec. \ref{AnalSNPMultiLoc}. 

From Table 5 in \cite{Supplementals:2017}, 
we see that AIC, BIC, and CV choose the one term model. Under conventional usage, $\widehat{ERM}_{1}$  chooses the model of size six. However, $\widehat{ERM}_{1}$ fluctuates between 7 and 8. Since we regard this as natural variability, we may instead be led to surmise $\widehat{ERM}_{1}$ is actually choosing the one term model. $\widehat{ERM}_{2}$ chooses the model of size 14, however,  $\widehat{ERM}_{2}$ fluctuates between 15 and 18 again suggesting a one term model may be appropriate.
$\hat{d}_{VC}$ chooses the model with 24 terms. The model chosen by both SCAD and ALASSO continues to be the one term model \eqref{DesignSCADSNIP}. Under this reasoning, 
all the methods are breaking down because $\hat{d}_{VC}$ does not have any obvious pattern.
Thus, again, this suggests the model list is inadequate.

It is natural to compare Tables 5 and 3 in \cite{Supplementals:2017}. The numerical behaviour of all six methods is identical and, though surprising, does not seem to be an error.  The same holds for SCAD and ALASSO. Thus, again, design variables seem not to influence model selection and $\hat{d}_{VC}$ provide the most credible model.

\section{Conclusions}
\label{GenreConcl}

A concise summary of the contributions in this paper is as follows.
Sec. \ref{sec:Ext:Vap:Bounds} presents the derivation of the objective function we used to estimate the VC dimension.  It is essentially an upper bound on the expected difference between two losses
that we have defined as $\Delta_m$ or $\Delta$, where the $m$ indicates the discretization of 
the loss function for a regression problem.   
In Subsec.  \ref{estimator} we give a procedure that uses our upper bound, 
nonlinear regression treating sample sizes as design points, a data driven
estimator of $\Delta_m$, two bootstrapping steps in sequence, and an
optimization over an arbitrary constant to form an 
of an estimator for the VC dimension.   While this sounds complex, in practice the
computations can usually be done in minutes on a regular laptop.

Even though we only have an upper bound on $\Delta_m$,
in Sec. \ref{ProofOfConsistency} we are able to give conditions under which our
estimator is consistent.  This is circumstantial evidence that our upper bound is tight
because otherwise it would be surprising to have consistency.

Most importantly, we have done an extensive comparison of our estimator of
VC dimension as a model selection method with seven established model selection methods,
namely two forms of empirical risk minimization, AIC, BIC, CV, and two sparsity criteria.
We did this for the special case of linear models but the same reasoning can be used for
any class of nonlinear models e.g., trees, for which the VC dimension can be identified.
We also gave one example of how our estimator for VC dimension performs better than the
original estimator in \cite{Vapnik:etal:1994}.   Other examples can be found in \cite{Merlin:etal:2017}.

More specifically, in Secs. \ref{chap:Numerical:Studies}, \ref{example}, and \ref{Application}, we 
used our technique
with simulated data, benchmark data, and data from a designed agronmic 
experiment, resp.   Since the actual model is known for simulated data, we used those examples to
verify our technique works well and to suggest guidelines for selecting design points.
When we used our method on two benchmark data sets we argued that, as in the simulations, using estimated VC dimension to do model selection gives better performance than 
the other standard methods.  More precisely, on these data sets, we argue that the results of our method
are as credible as the results from the other methods and in a few cases outperform the
established methods.  The key conceptual point here is that our method is geared towards model identification
via complexity whereas other methods such as CV, BIC, or sparsity methods seem to be geared
more to parsimony (even though CV has a prediction interpretation).  In terms of performance, 
our method resembles AIC.  However, AIC
has an optimality property under prediction which is conceptually different from model identification.  Indeed, the best predictor may not be the true model
when the true model has many small terms, so we think of our method not as a prediction-optimal
so much as able to detect the presence of small terms that other methods might miss.

The `real' data anayzed in Sec.  \ref{Application} is representative of a class of 
data that continues to be collected with
regularity in many agronomic (and other) settings.
Our analysis of this data is extensive first because this class of
data set includes two types of measured explanatory variables as well as variables defining the
design and, second, bcause we use this example to shed further light on what our method is doing.
We have two general conclusions.  First,  the design variables do not affect the model selection.
Otherwise put, the design variables do not affect the complexity of the data collected.   This seems
counterintuitive until one realizes that experimental design is used above all to increase the 
estimability and comparability of parameters and this goal is independent of data complexity
which the VC dimension assesses.   Second, although we have not formally assessed the robustness
of our method, it essentially never gave models that were wildly wrong even though it seemed more
sensitive to small effects.   We infer this because often the model chosen using VC dimension
was of moderate or large size within the model list while the other methods often gave one-term models.
A final observation is that the details of
these analyses suggested that the more data, or at least, more data types were present,
the higher the variability in estimates seemed to be.  This is counterintuitive since more
data is associated with more information and hence less variability.   A resolution to these
competing properties may be that the amount of data has not increased sufficently relative
to the complexity of the data generator that
the information in it has reduced the variability.  A related point is that our method `bails in' i.e., defaults to 
the largest model, more often than other methods which `bail out' i.e.,
default to the simplest model.  When this happens it suggests the model list is not
rich enough to provide good model identification.

There are two general issues that our work underscores: i) the variability in model selection can be 
high and ii) the relationship among data complexity, model complexity, and model list selection.

Treating these in turn, recall that under the usual regularity conditions, the rate of convergence 
of parameter estimators is ${\cal{O}}(1/\sqrt{n})$.  No such general statement can be made 
about model selection since convergence rates depend delicately on how close the chosen model is to the
true model (assuming it exists) and this in turn depends on the richness of the model space.  
Despite this, one can observe that, empirically, the rate of convergence
of chosen models to true models is slower than for parameter estimators to their true values (unless
the model space can be represented parametrically). 

For instance, in many of the examples here, the BIC chooses the
model with a single explanatory variable reflecting its tendency to parsimony since the number
of parameters is relatively heavily penalized.  However, this neglects the fact that in many cases the 
BIC value for a one term model and a two term model are very close, e.g., in Table 3 of
\cite{Supplementals:2017}, 
BIC(1 term) = -9160 and BIC(2 term) = BIC(3 term) = -9159 and one can argue that the
BIC decreases gradually until 30 explanatory variables are included in the model suggesting 30
is a sort of `breakpoint' in the nested sequence of models. In some sense, this takes variability 
into account and suggests 30 explanatory variables even though the best single model
has one term.  Of course, this is ad hoc and does not have any
justification in terms of the underlying variability of a model selection principle.
To get around  this, often an Occam's window (see \cite{occamwindow}) approach is taken i.e., 
rather than just picking the single model with the highest BIC a set of the form
$$
\left\{ m_\ell \mid \| m_\ell - m^* \| \leq t  \right\}
$$
is used for prediction, where $m_\ell$ denotes the model with $\ell$ 
explanatory variables, $m^*$ is the BIC 
optimal model, and $t$ is a threshold.  The same principles apply to AIC, and arguably to SCAD and 
ALASSO by defining neighborhoods in terms the decay parameter $\lambda$.  The limitation of
this approach is that while it is good for prediction, it's not clear how appropriate
it is for model identification.

Complexity based methods such as VC dimension, $\widehat{ERM}_{1}$, and $\widehat{ERM}_{2}$ have analogous properties even
though they are restricted to integers.  In these cases, it is harder to ascertain whether a change of,
say, 7 to 8 or the reverse, is particularly meaningful.  However, in some cases, the pattern of values of the
procedure was so strong e.g., often for CV, that regarding the values as chance fluctuations was unreasonable.
In all cases, when the rule for using a model selection procedure could not be reasonably followed, we tried to
infer what the method was trying to do by ignoring what seemed like extraneous variability.
Throughout our work, all of our comparisons have been motivated by the goal of selecting 
a single model even when a model averaging strategy would be better for predictive purposes.

We note that in many cases, the constancy, or near constancy, of the VC dimension, $\widehat{ERM}_{1}$, or $\widehat{ERM}_{2}$ indicates they are trying to encapsulate the concept of the complexity of the data, a different concept from the complexity of the model.  This supports the interpretation of them as reliable indicators of the 
complexity of the data generator.

Turning to the relationships between data and model complexity, we have noted that
typical design principles for agronomic data are intended to ensure parameters can be estimated
and compared effectively and this does not affect data complexity.
On the other hand, including an extra data type can increase the complexity of the data generator, 
as when SNP's were included in Subsecs. \ref{AnalSNP} and \ref{AnalPhenoSNPdsign}.  
We saw that when the complexity of the data exceeds the ability
of the models to encapsulate it, we can get bail-in and bail-out, where a model selection principle 
defaults to the biggest or smallest model, respectively.  
Choosing the smallest model reflects excess parsimony because the models can't cantain the data
while choosing the largest model suggests the models on the list are only able to capture a portion
of the information in the data.
One tries to avoid these cases by choosing a model list that has the true model, or a good approximation it it,
 somewhere in the
`interior' and at the same time is not too close to other models that might confuse a model
selection technique.  This permits robustness verifications and and other comparisons that are not possible if the chosen
model is on the `boundary' of the model list or has too many neighbors.   These principles for model list selection
seem to be under-appreciated.

The comparisons we have given of model selection techniques on the same data set under different conditions
seem to be relatively uncommon. However, such comparisons indicate important properties of model selection 
-- and model list selection.  Otherwise put, design principles for model
list selection need to be developed to complement those that already exist for model selection.
Accordingly, we call for more simulation
and real data comparisons of model selection principles using different model lists.

Finally, the immediate take home lesson from our 
VC dimension based approach
is that, unlike other model selection methods, it always gives useful answers. Typically, our method chose plausible 
models and, in the cases where it didn't, the results
suggested the model list was poorly chosen, in particular not complex enough.  However, in 
model selection, that, too, is a useful answer.

\acks{The authors gratefully acknowledge support from NSF grant $\#$ DMS-1419754 and invaluable computational support from the Holland Computing Center.
The authors also express their deep gratitude to two anonymous referees who gave us 
extensive guidance on how to improve our work.   Code and data for all the analyses discussed here can be found at
\url{https://github.com/
poudas1981/Wheat_data_set.}}


\newpage

\section{Supplementals to:  `Model Selection via the VC Dimension' }

Here we provide the figures and tables used in the main body of the paper but not presented there.   In order, we give the tables and figures for Secs. 2, 3, 5, and 6 in the main body of the paper.   The link where the data we used and sample code we used to analyze the data in Sec. 6.1.2 of the paper is \url{https://github.com/poudas1981/Wheat_data_set}.

\subsection{Tables and figures for Sec. 6}

\begin{figure}[ht]
    \caption{Interaction plot and boxplot of the Yield. Panel \ref{InteractionPlot} shows symbols for each variety over the 7 locations. The symbols are connected by lines to shows interaction. Panel \ref{BoxPlotYield} pools over varieties at each location.}
    \centering
    \begin{subfigure}[b]{0.5\textwidth}
        \centering
        \includegraphics[width = \textwidth]{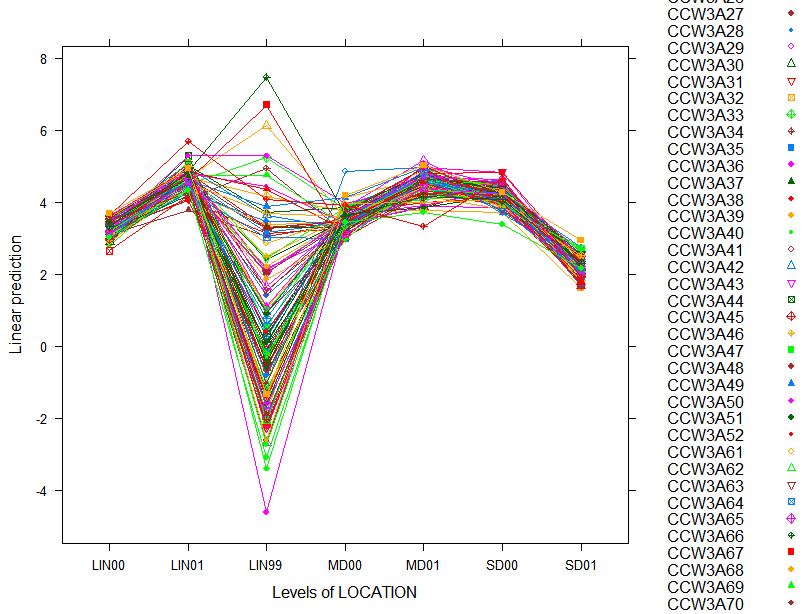}
        \caption{Interaction between Locations and Varieties}
        \label{InteractionPlot}
    \end{subfigure}
    \hfill
    \begin{subfigure}[b]{0.5\textwidth}
        \centering
        \includegraphics[width = \textwidth]{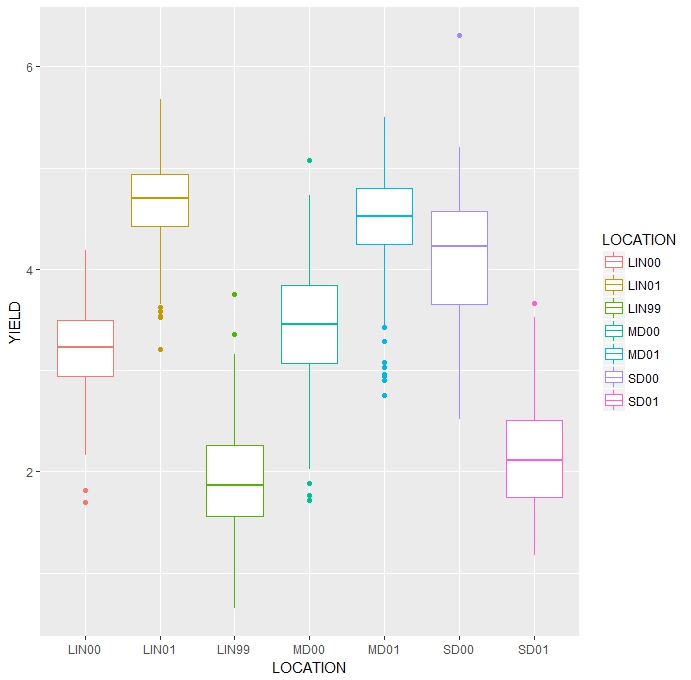}
        \caption{Boxplot of Yield by Location}
        \label{BoxPlotYield}
    \end{subfigure}
    \label{WheatPlotDesc}
\end{figure}

\begin{figure}[h]
    \caption{Scatter plots of YIELD vs phenotypic covariates}
    \centering
    \includegraphics[width=0.9\textwidth]{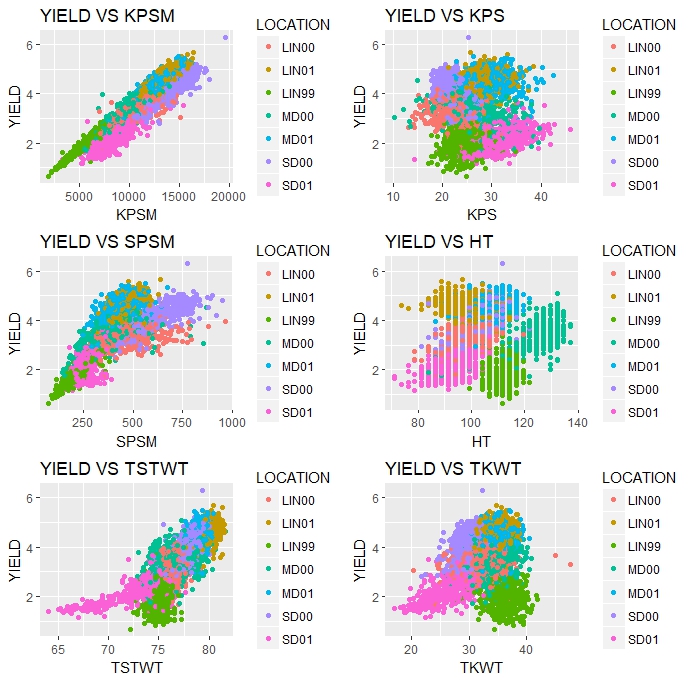}
    \label{fig:mesh1}
\end{figure}

\begin{table}[ht]
  \centering
  \caption{Correlation between phenotypic covariates over all locations}
    \begin{tabular}{cccccccc}
    \hline
  &YIELD & HT & TSTWT  & TKWT & SPSM & KPS & KPSM \\
  \midrule
  YIELD & 1.00 & 0.04 & 0.80 & 0.27 & $\bf{0.74}$ & 0.022 & 0.93\\
HT &    0.04 & 1.00 & 0.09 & 0.38 & 0.02 & -0.22 & -0.072\\
TSTWT & $\bf{0.80}$ &  0.09 &  1.00 & 0.53 & 0.53 & -0.06 &  0.64\\
TKWT &  0.27 &  0.38 &  0.53 &  1.00 & -0.11 & -0.09 & -0.09\\
SPSM &  $\bf{0.74}$ &  0.02 & 0.53 & -0.11 &  1.00 & -0.5 & $\bf{0.83}$\\
 KPS &   0.022 & -0.22 & -0.06 & -0.09 & -0.5 &  1.00 & 0.014\\
 KPSM  &  $\bf{0.93}$ & -0.07 &  0.64 & -0.09 &  $\bf{0.83}$ &  0.01 & 1.00\\
\hline
\end{tabular}
\label{CorrPheno}
\end{table}

\clearpage

\newpage 
\subsection{Figures from Sec. 6.2.1}


\begin{figure}[ht]
    \caption{Scatter plots of YIELD vs phenotypic covariates in Lincoln 1999}
    \centering
    \includegraphics[width=0.8\textwidth]{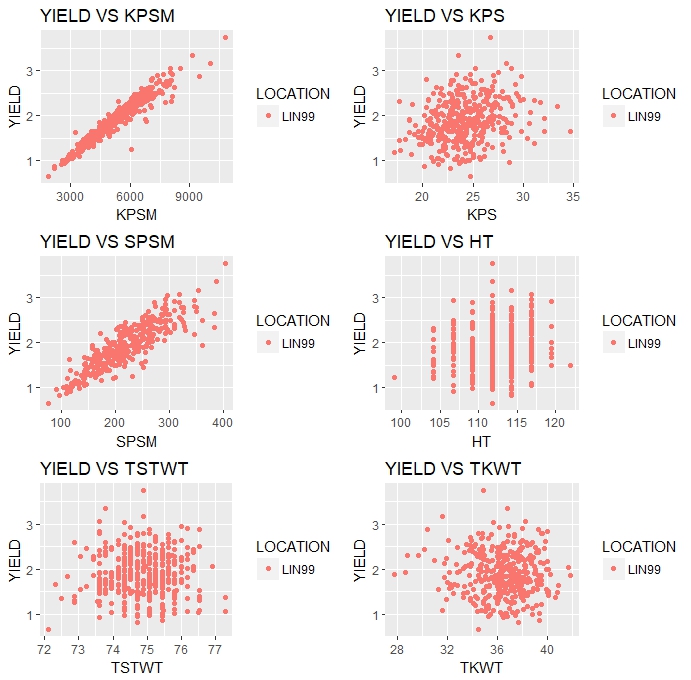}
    \label{fig:mesh1}
\end{figure}

\begin{figure}[ht]
    \caption{Scatter plots of YIELD vs phenotypic covariates in Lincoln 1999}
    \centering
    \includegraphics[width=0.7\textwidth]{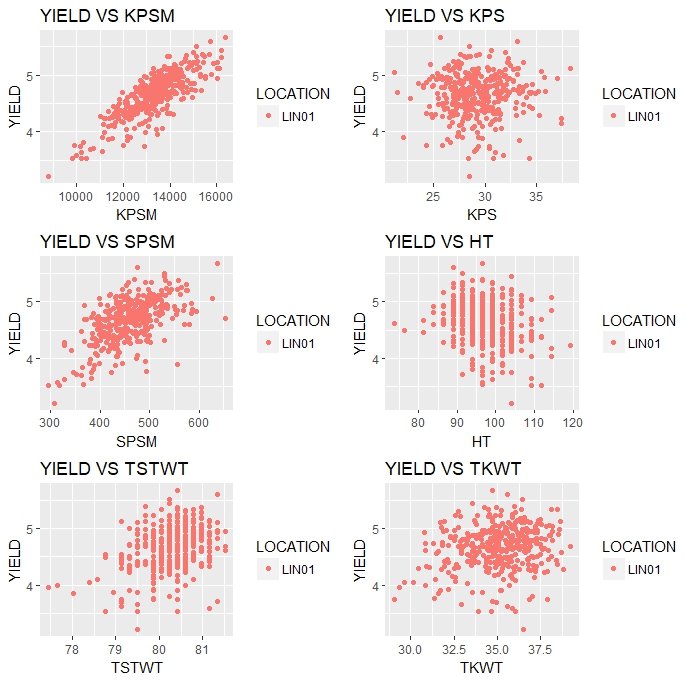}
    \label{fig:mesh1}
\end{figure}

\clearpage

\subsection{Table for Sec. 6.4.1}
\begin{table}[h]
    \caption{The column labeled size gives the number of coefficients for each of the linear model. 
    The $2^{nd}$ through $7^{th}$ columns give the corresponding estimates for $\hat{h}$, $\widehat{ERM}_{1}$, $\widehat{ERM}_{2}$, $AIC$, $BIC$ and CV for $\sf{Wheat}$ in Lincoln 2001 using phenotypic and SNP variables.}
    \begin{subtable}[h]{0.45\textwidth}
        \centering
        \begin{tabular}{ccccccc}
          Size& $\hat{d}_{VC}$ & $\widehat{ERM}_{1}$ & $\widehat{ERM}_{2}$ & AIC & BIC & CV  \\
        \midrule
        1&  40&    \bf{4}&   12& \bf{-1055}& \bf{-1043}& \bf{0.00352}\\
        2&  40&    4&   12& -1053& -1037& 0.00352\\
        3&  40&    4&   12& -1051& -1032& 0.00353\\
        4&  39&    4&   11& -1050& -1026& 0.00353\\
        5&  34&    4&   \bf{10}& -1048& -1020& 0.04558\\
        6&  34&    4&   10& -1046& -1015& 0.04206\\
        7&  34&    4&   10& -1044& -1009& 0.04195\\
        8&  36&    4&   11& -1042& -1003& 0.04272\\
        9&  37&    4&   11& -1040&  -997& 0.04294\\
        10& 36&    4&   11& -1038&  -991& 0.04350\\
        11& 36&    4&   11& -1043&  -992& 0.07936\\
        12& 37&    4&   11& -1042&  -987& 0.08302\\
        13& 36&    4&   11& -1040&  -981& 0.08296\\
        14& 37&    4&   11& -1041&  -979& 0.06465\\
        15& 35&    4&   11& -1040&  -973& 0.06295\\
        16& 37&    4&   11& -1038&  -968& 0.06338\\
        17& 37&    4&   11& -1036&  -962& 0.06401\\
        18& 35&    4&   11& -1035&  -956& 0.06202\\
        19& 36&    4&   11& -1033&  -951& 0.06527\\
        20& 36&    4&   11& -1031&  -945& 0.06451\\
        21& 34&    4&   10& -1029&  -939& 0.06535\\
        22& 35&    4&   11& -1028&  -933& 0.06692\\
        23& 37&    4&   11& -1026&  -928& 0.06749\\
        24& 35&    4&   11& -1024&  -922& 0.06673\\
        25& 37&    4&   11& -1023&  -917& 0.06469\\
        26& 37&    4&   11& -1023&  -913& 0.06055\\
        27& 37&    4&   11& -1021&  -907& 0.06135\\
        28& 37&    4&   11& -1021&  -903& 0.05328\\
        29& 37&    4&   11& -1020&  -899& 0.03878\\
        30& 36&    4&   11& -1019&  -893& 0.04111\\
        31& 34&    4&   10& -1018&  -888& 0.04418\\
        32& 35&    4&   10& -1042&  -908& 0.01195\\
        33& \bf{35}&    4&   10& -1042&  -908& 0.01195
        \end{tabular}
        \label{LinSNP101}
    \end{subtable}
\end{table}

\newpage
\subsection{Table for Sec. 6.4.2}
\begin{table}[h]
    \caption{The column labeled size gives the number of coefficients for each of the linear model. 
    The $2^{nd}$ through $7^{th}$ columns give the corresponding estimates for $\hat{h}$, $\widehat{ERM}_{1}$, $\widehat{ERM}_{2}$, $AIC$, $BIC$ and CV for $\sf{Wheat}$ for multi-location analysis using phenotypic and SNP variables.}
    \begin{subtable}[h]{0.45\textwidth}
        \centering
        \begin{tabular}{ccccccc}
          Size& $\hat{d}_{VC}$ & $\widehat{ERM}_{1}$ & $\widehat{ERM}_{2}$ & AIC & BIC & CV \\
        \midrule
        1&  33&    8&   16& \bf{-9142}& \bf{-9160}& \bf{0.001801809}\\
        2&  33&    8&   16& -9136& -9159& 0.001802400\\
        3&  33&    8&   16& -9129& -9159& 0.001802509\\
        4&  33&    8&   16& -9122& -9157& 0.001804594\\
        5&  33&    8&   16& -9114& -9156& 0.001806903\\
        6&  32&    \bf{7}&   16& -9109& -9156& 0.001806764\\
        7&  33&    8&   16& -9101& -9154& 0.001807267\\
        8&  33&    8&   16& -9095& -9153& 0.001811197\\
        9&  33&    8&   16& -9087& -9151& 0.001814232\\
        10& 32&    7&   16& -9081& -9152& 0.001815581\\ 
        11& 32&    7&   16& -9073& -9150& 0.001820324\\
        12& 32&    7&   16& -9067& -9149& 0.001820803\\
        13& 32&    7&   16& -9061& -9149& 0.001822245\\
        14& 31&    7&   \bf{15}& -9053& -9147& 0.001822267\\
        15& 31&    7&   15& -9045& -9145& 0.001823113\\
        16& 31&    7&   15& -9042& -9148& 0.001822037\\
        17& 43&    8&   18& -9035& -9146& 0.001827415\\
        18& 43&    8&   18& -9027& -9144& 0.001828095\\
        19& 43&    8&   18& -9019& -9142& 0.001828721\\
        20& 31&    7&   15& -9012& -9141& 0.001828733\\
        21& 32&    7&   16& -9004& -9139& 0.001835135\\
        22& 32&    7&   16& -8996& -9137& 0.001838472\\
        23& 32&    7&   16& -8989& -9136& 0.001838386\\
        24& \bf{31}&    7&   15& -8982& -9134& 0.001846458\\
        25& 43&    8&   18& -8978& -9137& 0.001844896\\
        26& 42&    8&   18& -8974& -9138& 0.001843160\\
        27& 42&    8&   18& -8966& -9137& 0.001843012\\
        28& 41&    8&   17& -8959& -9135& 0.001843302\\
        29& 41&    8&   17& -8951& -9133& 0.001847333\\
        30& 41&    8&   17& -8944& -9132& 0.001848343\\
        31& 40&    8&   17& -8944& -9138& 0.001847684\\
        32& 43&    8&   18& -8939& -9139& 0.001850055
        \end{tabular}
        \label{Gene1}
    \end{subtable}
\end{table}

\newpage

\subsection{Table for Sec 6.5.1}

\begin{table}[h]
    \caption{The column labeled size gives the number of coefficients for each of the linear model. 
    The $2^{nd}$ through $7^{th}$ columns give the corresponding estimates for $\hat{h}$, $\widehat{ERM}_{1}$, $\widehat{ERM}_{2}$, $AIC$, $BIC$ and CV for $\sf{Wheat}$ in Lincoln 2001 using phenotypic, SNP and design variables.}
    \begin{subtable}[h]{0.45\textwidth}
        \centering
        \begin{tabular}{cccccccc}
          Size& $\hat{d}_{VC}$ & $\widehat{ERM}_{1}$ & $\widehat{ERM}_{2}$ & AIC & BIC & CV \\
        \midrule
        1&  39&    \bf{4}&   11& \bf{-1056}& \bf{-1040}& \bf{0.003523772}\\
        2&  39&    4&   11& -1054& -1034& 0.003523852\\
        3&  39&    4&   11& -1052& -1028& 0.003528050\\
        4&  39&    4&   11& -1051& -1023& 0.003531868\\
        5&  35&    4&   11& -1049& -1017& 0.045582343\\
        6&  38&    4&   11& -1047& -1012& 0.042056585\\
        7&  34&    4&   \bf{10}& -1045& -1006& 0.041950879\\
        8&  36&    4&   11& -1044& -1000& 0.042716568\\
        9&  36&    4&   11& -1042&  -995& 0.042944324\\
        10& 35&    4&   11& -1040&  -989& 0.043503630\\
        11& 36&    4&   11& -1043&  -988& 0.079361885\\
        12& 35&    4&   11& -1042&  -983& 0.083019021\\
        13& 35&    4&   11& -1040&  -977& 0.082959969\\
        14& 37&    4&   11& -1041&  -975& 0.064653613\\
        15& 35&    4&   11& -1040&  -969& 0.062946594\\
        16& 36&    4&   11& -1039&  -964& 0.063382297\\
        17& 36&    4&   11& -1037&  -958& 0.064013291\\
        18& 37&    4&   11& -1035&  -953& 0.062024613\\
        19& 34&    4&   10& -1034&  -947& 0.065271013\\
        20& 36&    4&   11& -1032&  -942& 0.064513608\\
        21& 35&    4&   11& -1030&  -936& 0.065345476\\
        22& 36&    4&   11& -1028&  -930& 0.066919227\\
        23& 36&    4&   11& -1027&  -924& 0.067491940\\
        24& 37&    4&   11& -1025&  -919& 0.066729472\\
        25& 35&    4&   11& -1024&  -914& 0.064685020\\
        26& 35&    4&   11& -1024&  -910& 0.060548876\\
        27& 36&    4&   11& -1022&  -904& 0.061349332\\
        28& 34&    4&   10& -1022&  -900& 0.053281068\\
        29& 34&    4&   10& -1022&  -896& 0.038782094\\
        30& 38&    4&   11& -1020&  -891& 0.041109422\\
        31& 34&    4&   10& -1019&  -885& 0.044183504\\
        32& 35&    4&   10& -1040&  -903& 0.011947595\\
        33& \bf{35}&    4&   10& -1040&  -903& 0.011947595\\
        \end{tabular}
        \label{phenosnpdesiGene12001}
    \end{subtable}
\end{table}

\newpage
\subsection{Table for Sec. 6.5.2}

\small{
\begin{table}[h]
    \caption{The column labeled size gives the number of coefficients for each of the linear model. 
    The $2^{nd}$ through $6^{th}$ columns give the corresponding estimates for $\hat{h}$, $\widehat{ERM}_{1}$, $\widehat{ERM}_{2}$, AIC, BIC, and CV for $\sf{Wheat}$ for multi-location using phenotypic, SNP and design variables.}
    \begin{subtable}[h]{0.45\textwidth}
        \centering
        \begin{tabular}{ccccccc}
          Size& $\hat{h}$ & $\widehat{ERM}_{1}$ & $\widehat{ERM}_{2}$ & AIC & BIC & CV \\
        \midrule
        1&  33&    8&   16& \bf{-9142}& \bf{-9160}& \bf{0.001801809}\\
        2&  33&    8&   16& -9136& -9159& 0.001802400\\
        3&  33&    8&   16& -9129& -9159& 0.001802509\\
        4&  33&    8&   16& -9122& -9157& 0.001804594\\
        5&  33&    8&   16& -9114& -9156& 0.001806903\\
        6&  32&    \bf{7}&   16& -9109& -9156& 0.001806764\\
        7&  33&    8&   16& -9101& -9154& 0.001807267\\
        8&  33&    8&   16& -9095& -9153& 0.001811197\\
        9&  33&    8&   16& -9087& -9151& 0.001814232\\
        10& 32&    7&   16& -9081& -9152& 0.001815581\\
        11& 32&    7&   16& -9073& -9150& 0.001820324\\
        12& 32&    7&   16& -9067& -9149& 0.001820803\\
        13& 32&    7&   16& -9061& -9149& 0.001822245\\
        14& 31&    7&   \bf{15}& -9053& -9147& 0.001822267\\
        15& 31&    7&   15& -9045& -9145& 0.001823113\\
        16& 31&    7&   15& -9042& -9148& 0.001822037\\
        17& 43&    8&   18& -9035& -9146& 0.001827415\\
        18& 43&    8&   18& -9027& -9144& 0.001828095\\
        19& 43&    8&   18& -9019& -9142& 0.001828721\\
        20& 31&    7&   15& -9012& -9141& 0.001828733\\
        21& 32&    7&   16& -9004& -9139& 0.001835135\\
        22& 32&    7&   16& -8996& -9137& 0.001838472\\
        23& 32&    7&   16& -8989& -9136& 0.001838386\\
        24& \bf{31}&    7&   15& -8982& -9134& 0.001846458\\
        25& 43&    8&   18& -8978& -9137& 0.001844896\\
        26& 42&    8&   18& -8974& -9138& 0.001843160\\
        27& 42&    8&   18& -8966& -9137& 0.001843012\\
        28& 41&    8&   17& -8959& -9135& 0.001843302\\
        29& 41&    8&   17& -8951& -9133& 0.001847333\\
        30& 41&    8&   17& -8944& -9132& 0.001848343\\
        31& 40&    8&   17& -8944& -9138& 0.001847684\\
        32& 43&    8&   18& -8939& -9139& 0.001850055
        \end{tabular}
        \label{multiphenosnpdesiGene12001}
    \end{subtable}
\end{table}

}

\bibliography{references}

\end{document}